\documentclass[11pt]{article}
\usepackage{amsmath,amssymb,amsthm,mathrsfs}
\usepackage{graphicx}
\usepackage{xcolor}
\usepackage{booktabs}
\usepackage{multirow}
\usepackage[authoryear,round]{natbib}
\usepackage{hyperref}
\hypersetup{
  colorlinks=true,
  linkcolor=blue,
  citecolor=blue,
  urlcolor=blue
}

\theoremstyle{definition}
\newtheorem{lemma}{Lemma}[section]
\newtheorem{proposition}{Proposition}[section]
\newtheorem{theorem}{Theorem}[section]

\newtheorem{example}{Example}[section]
\newtheorem{corollary}{Corollary}[section]
\newtheorem{remark}{Remark}[section]

\newcommand{\RR}{\mathbb{R}}

\newcommand{\PP}{\mathbb{P}}

\def\Real{{\mathbb R}}
\def\SV{\mathop{\it SV}\nolimits}

\def\Prob{{\mathbb P}}

\def\F{{\mathcal F}}

\def\conv{\mathop{\rm conv}\nolimits}

\def\vhalf{\textstyle{\frac 1 2}}
\def\sign{\mathop{\rm sign}\nolimits}

\def\given{\mid}

\def\E{{\mathcal E}}

\def\Xmax{X_{(n)}}
\def\Zmax{Z_{(n)}}

\title{Non-standard boundary behaviour in two-component mixture models}
\author{
  Heather Battey\thanks{Department of Mathematics, Imperial College London}, \quad 
  Peter McCullagh\thanks{Department of Statistics, University of Chicago} \quad and \quad 
  Daniel Xiang$^{\dagger}$
}
\date{\today}

\begin{document}

\maketitle

\begin{abstract}

	The empirical Bayes approach to multiple testing, widely used in modern statistical contexts, is frequently formulated as a two-component mixture model of the form $F_\theta = (1-\theta)F_0 + \theta F_1$, where $F_0$ is standard Gaussian and $F_1$ is a completely specified heavy-tailed distribution with the same support. For a sample of $n$ independent and identically distributed values $X_i \sim F_\theta$, the maximum likelihood estimator $\hat\theta_n$ is asymptotically normal provided that $0 < \theta < 1$ is an interior point.  This paper investigates the large-sample behaviour for boundary points, which is entirely different and strikingly asymmetric for $\theta=0$ and $\theta=1$.  The reason for the asymmetry has to do with typical choices such that $F_0$ is an extreme boundary point and $F_1$ is usually not extreme. On the right boundary, well known results on boundary parameter problems are recovered, giving $\lim \Prob_1(\hat\theta_n < 1)=1/2$. On the left boundary, $\lim\Prob_0(\hat\theta_n > 0)=1-1/\alpha$, where $1\leq \alpha \leq 2$ indexes the domain of attraction of the density ratio $f_1(X)/f_0(X)$ when $X\sim F_0$. For $\alpha=1$, which is the most important case in practice, we show how the tail behaviour of $F_1$ governs the rate at which $\Prob_0(\hat\theta_n > 0)$ tends to zero. A new limit theorem for the joint distribution of the sample maximum and sample mean conditional on positivity establishes multiple inferential anomalies. Most notably, given $\hat\theta_n > 0$, the likelihood ratio statistic has a conditional null limit distribution $G\neq\chi^2_1$ determined by the joint limit theorem.  We show through this route that no advantage is gained by extending the single distribution $F_1$ to the nonparametric composite mixture generated by the same tail-equivalence class. \\

\bigskip

\noindent \emph{Some key words:} $\alpha$-stable limit law; 
likelihood-ratio statistic; 
local false discovery rate;
multiple testing;
non-standard likelihood theory;
regular variation;
tail behaviour.
\end{abstract}

\section{Introduction}\label{secIntro}

Let $F_0, F_1$ be distinct probability distributions on the same measurable space.
For each $0 \le\theta \le 1$, the mixture
\begin{equation}\label{eqMixtureModel}
	F_\theta = (1 - \theta) F_0 + \theta F_1
\end{equation}
is also a probability distribution on the same space.
In this paper, the phrase \lq\lq mixture model generated by $F_0, F_1$\rq\rq\ is interpreted as the set of convex combinations
\[
\conv(F_0, F_1) = \{ F_\theta : 0 \le \theta \le 1\}
\]
in which the generators $F_0, F_1$ are the boundary points.

The model \eqref{eqMixtureModel} arises particularly in connection with the empirical Bayes approach to multiple testing \citep{Efron2001}, 
where interest is either in evaluating the simultaneous correctness of a set of statements concerning the same null hypothesis, or in the assessment of different null hypotheses, at most a small number of which are false. The case in which one of the components is standard Gaussian reflects a reduction by sufficiency to a set of pivotal test statistics and is particularly relevant for the applications we have in mind. The other component is typically chosen to have tails that are heavier than Gaussian, for instance a standard Gaussian convolution corresponding to a signal plus noise model at the relevant site.  The interpretation of \emph{site} is context-specific and may refer to a genomic locus in a genetics context or an energy level in a particle physics context. The response at site~$i$ is treated as a draw from $F_\theta$, which means $F_1$ with probability~$\theta$ and $F_0$ with probability $1-\theta$, on account of the label being unknown.

The goal of this paper is to study the large-sample behaviour of the maximum-likelihood estimator and related statistics
when $F_0$ and $F_1$ have the same support and the data are generated independently according to the boundary distribution $X_{1},\ldots,X_n\sim F_0$. The statistical relevance of such considerations is context dependent but in general reflects the usual analogue of proof by contradiction underpinning standard hypothesis tests of null effects.

Under the model \eqref{eqMixtureModel}, $X_1,\ldots, X_n$ are independent random variables with distribution~$F_\theta$ for some~$0\le\theta\le 1$.
Standard theory for maximum-likelihood estimators tells us that, under suitable regularity conditions,
$n^{1/2} (\hat\theta_n - \theta)$ has a zero-mean Gaussian limit.
A principal regularity condition is that all distributions in the model have the same support. This condition need not be satisfied by $F_0, F_1$, but it is automatic for the sub-model in which $0 < \theta < 1$. Whether~$F_0$ and~$F_1$ have the same support or not, a second regularity condition requires $\theta$ to be an interior point. Otherwise, if $\theta=0$, the event $\hat\theta_n = 0$ has positive probability, so the limit distribution cannot be Gaussian. The mixture model can therefore be regular only if the boundary points are excluded. From a list of further regularity conditions, there is the possibility that more than one is violated. In the empirical Bayes formulations that motivated the present work, infinite variance of the density ratio $f_1(X)/f_0(X)$ presents considerable challenges not covered by existing literature on boundary inference problems.

In the standard case, the generators share a common support and have positive densities $f_0(x), f_1(x)$,
so that $F_\theta$ has density $f_\theta(x) = (1-\theta)f_0(x) + \theta f_1(x)$, which is linear in~$\theta$.
The contribution $\log f_\theta(x)$ to the log likelihood from a single observation at~$x$ is strictly concave as a function of~$\theta$,
and the sum $\sum_i \log f_\theta(x_i)$ is also strictly concave.
Consequently, the maximum-likelihood estimate is either a stationary point or a boundary point.

The log likelihood derivatives are 
\begin{align}
\begin{split}\label{eqDeriv}
l'(\theta) &=  \sum_{i=1}^n \frac{h(X_i) - 1} {1-\theta + \theta h(X_i)}  \\
l''(\theta) &= -\sum_{i=1}^n \Bigl(\frac{h(X_i) - 1} {1-\theta + \theta h(X_i)}\Bigr)^2 \leq 0, 
\end{split}
\end{align}
and $l'(0) = \sum (h(X_i) - 1)$, where $h(x) = f_1(x)/f_0(x)$ is the density ratio of the boundary points. It follows that the event $\hat\theta_n > 0$ is the same as the event $n^{-1} \sum h(X_i) > 1$, where $X_1,\ldots, X_n$ are independently distributed as $F_0$. In other words, $\hat\theta_n > 0$ if and only if the sample average of the transformed variables $h(X_i)$ exceeds its expected value
\[
E (h(X)) = \int h(x) f_0(x)\, dx = 
\int f_1(x)\, dx = 1.
\]
If $h(X)$ has finite variance $\sigma^2$, then $n^{-1/2}l'(0)$ is zero-mean Gaussian for large $n$, and $-n^{-1}l''(0) \rightarrow \sigma^2$ by the law of large numbers. In this case $\hat\theta_n=0$ with probability 1/2 in the large-sample limit; conditionally on $\hat\theta_n>0$, the random variable $n^{1/2}\hat\theta_n>0$ is distributed half-Gaussian with scale parameter~$\sigma$. This is the familiar boundary-parameter result established by other authors \citep[e.g.][]{Chernoff1954,SelfLiang1987,Geyer1994}. The boundary probabilities exhibit more interesting behaviours in the cases for which $h(X)$ does not have a finite second moment under the null model. It is those situations that we explore here.

\section{Preliminary insights and outline of results}\label{secPrelim}

Here and henceforth, $\Prob_0(\cdot)$ denotes the probability with respect to the null generator or its $n$-fold product $F_0^{\otimes n}$. The Landau notation $a(n)\sim b(n)$ means 
\[
\lim_{n\rightarrow \infty} \frac{a(n)}{b(n)} = 1,
\]
or equivalently $a(n)=b(n)(1+o(1))$ for positive functions $a(n), b(n)$. 

The motivating example for a large part of this paper is a restricted class of two-component mixtures in which
the density ratio is an even function that is continuous, unbounded and ultimately monotone.  Ultimate monotonicity means that to each $\eta$ sufficiently large there corresponds a number $\xi$ such that
\begin{equation}\label{umc}
\{x : h(x) > \eta\} = \{x : |x| > \xi\}.
\end{equation}
Mills's approximation to the Gaussian tail probability yields
\[
\Prob_0(|X| > \xi) = \frac{2\phi(\xi)}{\xi}\bigl(1+ O(\xi^{-2})\bigr).
\]
 On using the implicit definition $\phi(\xi) = \eta^{-1} f_1(\xi)$,
\begin{equation}\label{null_tail}
\Prob_0(h(X) > \eta) = \Prob_0(|X| > \xi) \sim 2\xi^{-1} \phi(\xi) 
=\frac{2 f_1(\xi)} {\eta\, \xi}.
\end{equation}
It follows from \eqref{null_tail} that the asymptotic inverse relationship $\xi(\eta) = h^{-1}(\eta)$ determines the null tail behaviour via $f_1(\xi)/\xi$,  
and thereby the limit distribution of normalized sums.

Let $L: \RR^+ \rightarrow \RR^+$ be a slowly-varying function, in the sense that, for all $k>0$
\begin{equation}\label{eqSlow}
\lim_{x\rightarrow \infty} \frac{L(kx)}{L(x)}=1.
\end{equation}
For a large class of non-null generators,  including all whose density satisfies
\begin{equation}\label{tail_index_class}
-\log f_1(x) \sim |x|^{2\kappa} L(x)
\end{equation}
for some $0 \le \kappa < 1$, we find that $\xi(\eta) \sim \sqrt{2\log \eta}$ as $\eta\to\infty$.
This asymptotic inverse implies
\[
\Prob_0(h(X) > \eta) \sim \frac{2f_1(\xi(\eta))} {\eta \sqrt{2\log \eta}} \sim \frac 1 {\eta\, L_1(\eta)},
\]
where, for every $f_1$ in the class \eqref{tail_index_class}, $L_1$~is also slowly varying. In all such cases, the random variable $h(X)$, or more correctly, its distribution, 
belongs to the Cauchy domain of attraction with index $\alpha=1$, meaning that suitably standardized sums converge in distribution to a random variable in the Cauchy class. Thus, the Cauchy class is virtually the default in practical
work with mixtures where the null generator is Gaussian and $F_1$ is symmetric with heavier tails. In such situations, it is important for the appropriate calibration of statistical inference to characterize the rate at which the error tends to zero.

To outline the nature of our results, if $F_1$~has regularly-varying tails with index $-2\delta$, i.e., the upper tail probability $\bar{F}_1(x)=\Prob_1(X>x)$ satisfies 
\[
\lim_{x\rightarrow \infty} \frac{\bar{F}_1(tx)}{\bar{F}_1(x)}=t^{-2\delta}
\]
for $t > 0$, then the rate is logarithmic:
\[
\Prob_0(\hat\theta_n > 0) \sim\frac {\delta} {\log n}.
\]
By contrast, if $F_1$ has exponential tails, e.g.,~$-\log \bar F_1(x) \sim |x|^{2\gamma}$ for large~$x$ and $0 < \gamma < 1$, then
\[
\Prob_0(\hat\theta_n > 0) \sim\frac {\gamma\, 2^\gamma} {(\log n)^{1-\gamma}}.
\]

While less important for statistical applications, we also show that if $h(X_i)$ belongs to the the domain of attraction of any non-Cauchy class, which necessarily has index $1 < \alpha \le 2$, $\lim \Prob_0(\hat\theta_n > 0) = 1 - \alpha^{-1}$, which is strictly positive but not more than one half. A proof is given at the end of Section \ref{subsec:stabilizing-seq}.

The limit distribution of the likelihood ratio statistic conditional on $\hat\theta_n > 0$ is obtained in \S \ref{L-R_limit} using a non-quadratic local approximation. The derivation is based on a new limit theorem establishing a large-sample joint distribution for the sample mean and sample maximum of a random variable in the Cauchy domain of attraction, conditional on positivity of the sample mean.

\section{Existing literature}

The inferential problem for two-component mixtures belongs to a class of boundary problems for which an extensive literature was carefully surveyed by \citet{Brazzale}. Appendix A.1 of that work outlines an argument, due to \citet{SelfLiang1987}, establishing the limit distribution of the log likelihood-ratio statistic when the true value of the parameter is on the boundary. Two aspects of the argument are problematic when $h(X)$ does not have finite variance under the null distribution $F_0$: that the Fisher information at $\theta=0$ is not finite, this being the expectation of $\sum (h(X_i)-1)^2$; and that a suitably rescaled version of $l'(0)$ is not asymptotically normally distributed. The same issues afflict the argument of \citet{GS1985} \citep[see][Appendix A.3]{Brazzale}.

In the context of a mixture model with two unit-Gaussian components, one having unknown mean, \citet{Bickel} and \citet{LiuShao} established that the likelihood ratio statistic diverges to infinity at rate $O(\log(\log n))$ and that the asymptotic null distribution of a suitably standardized version is of extreme-value type. The conclusions of the present paper are strikingly different: that the likelihood ratio statistic converges to zero under the null model and that its conditional distribution given positivity has a limit distribution~$G$ (Theorem \ref{LR_limit}) that is visibly close to but not exactly $\chi_1^2$. 

Gaussian-Gaussian mixtures of unequal means and equal variances lead to density ratios $h(X)$ that are log-linear in $X$ and thus in the normal domain of attraction. The extreme-value limit distribution emerges from estimation of the unknown mean parameter in the non-null component. In recent work, \citet{ShiDrton} have studied a split-sample version of the likelihood ratio statistic for the mixture problem with unequal means, following the generic \emph{universal inference} construction of \citet{WRB}. In well-behaved parametric problems, universal inference violates sufficiency and is inferior to a standard likelihood-based analysis. Its advantage in the two-component mixture setting is that sample splitting breaks some dependencies in the likelihood ratio construction, allowing \citet{ShiDrton} to show that the asymptotic null distribution of the split-sample likelihood ratio statistic is standard Gaussian after centering and scaling. In our case, the source of the difficulties is different and is not evaded by sample splitting. 

An important conceptual difference between the models is that the signal in \citet{ShiDrton} is an arbitrary fixed real number corresponding to the mean of the non-null distribution. By contrast, the signal in our formulation is a random variable whose realizations are centered at zero and appreciably large with non-negligible probability. 

\citet{LiChenMarriot2009} and \citet{ChenLi2009} allow two-component mixture constructions of the form $(1-\theta)f(x;\lambda_0) + \theta f(x;\lambda_1)$ such that the Fisher information at $\theta=0$ is not finite. However they study a different problem based on a penalized likelihood function that forces the estimate $\hat\theta$ away from 0 and 1. Their implicit assessment of homogeneity is thus based on $\lambda_0=\lambda_1$ rather than $\theta=0$.

Contrasted with \citet{GS1985}, \citet{SelfLiang1987} and \citet{Bickel}, non-existence of second moments necessitates a radically different approach based on the theory of $\alpha$-stable limits \citep{GK1954} and regular variation \citep{Bingham}. From this we are able to delineate the role of tail properties in determining the type-one error rate $\Prob_0(\hat\theta_n > 0)$ and the anomalous limiting behaviour of likelihood-based statistics.

\section{Limit distributions for sums}\label{secLimitDist}

\subsection{Introduction}

From the discussion of \S \ref{secIntro}, the event $\hat\theta_n >0$ is equivalent to the event that the sample average of the density ratios $h(X_i)$ exceeds its expectation. The large-sample boundary behaviour $\PP_0(\hat\theta_n >0)$ thus requires some theory of limit distributions for sums of iid random variables with non-finite variance.

The results of this section are stated in terms of a generic random variable $X$. This is both for notational convenience and for ease of application of the results beyond the Gaussian mixture setting for which they were initially conceived. For the statistical questions we have in mind, addressed in \S \ref{secImplications}, the results are applied with $h(X)$ or $h(X)-1$ in place of $X$, using the ultimate monotonicity argument explained in \S \ref{secPrelim}

\subsection{Stable limits}\label{secStable}
The theory of limit distributions for the sum of independent and identically distributed random variables is tied up with stability of convolutions.  Modulo affine transformation, every distribution that has a convolution limit is associated with a pair $(\alpha,\beta)$, and is said to be in the domain of attraction of the $(\alpha, \beta)$ stable law;  in essence, the set of Borel distributions on the real line is partitioned into equivalence classes. There are also non-degenerate distributions that do not belong to the domain of attraction of any stable class, for example, $\bar F(x) \sim 1/\log x$.

Every stable distribution with $\alpha \ge 1$ has a density that is strictly positive on the real line; only in a few cases is it possible to express the density in terms of standard functions. However, the characteristic function of every limit distribution is necessarily of the form $\psi(a + b t)$ where
\[
\log \psi(t) =  \left\{ \begin{array}{ll} 
- t^2 & \quad(\alpha = 2), \\
-|t|\bigl(1 +   i\beta \sign(t) \frac{2}{\pi}\log|t| \bigr)  & \quad (\alpha = 1),\\
-|t|^\alpha \bigl(1 - i \beta \sign(t) \tan(\pi\alpha/2) \bigr) & \quad (\alpha \neq 1).
\end{array} \right.
\]
See \citet[][Theorem~8.3.2]{Bingham} for a statement of this result, or
\citet[][chapter~34]{GK1954} where the sign of~$\beta$ is mistakenly reversed for $\alpha\neq 1$ \citep[see comment I.7 on page 253 of][]{Zolotarev}. 

Note that not all $(\alpha,\beta)$-combinations give rise to distinct distributions.
In the Gaussian case ($\alpha = 2$) the last version of the log characteristic function reduces to $-t^2$, so $\beta$ is immaterial and all limits are symmetric.
However, limits in every other class are symmetric only if $\beta=0$.

\begin{figure}
	\begin{center}
		\includegraphics[width=0.65\linewidth]{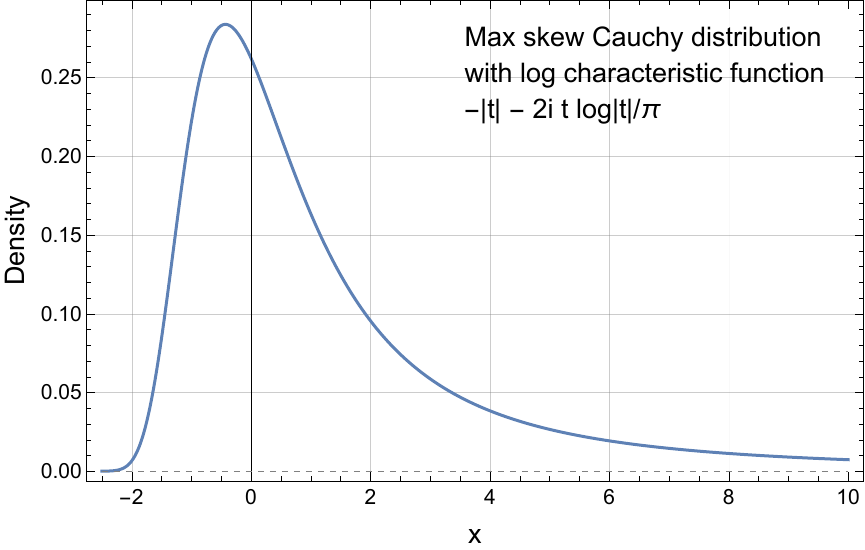}
		\caption{Density function of the maximally-skew Cauchy distribution}
	\end{center}
\end{figure}

\citet[][section~2.5]{Zolotarev} gives the tail behaviour of the $(\alpha,\beta)$-limit distribution parameterized according to the characteristic function shown above:
\[
P_{\alpha,\beta}(X > x) \sim (1 + \beta) C_\alpha x^{-\alpha};\quad
P_{\alpha,\beta}(X < - x) \sim (1 - \beta) C_\alpha x^{-\alpha},
\]
where $C_\alpha = \pi^{-1} \Gamma(\alpha) \sin(\pi\alpha/2)$ and $C_1=\pi^{-1}$.
See also \citet[][section~XVII.6]{Feller}. For $\beta = -1$, the first limit is interpreted as $x^\alpha P_{\alpha,\beta}(X > x) \to 0$ as $x \to \infty$;
likewise for $\beta=1$ in the second limit. Note that $C_2 = 0$, so this characterization $P_{2,\beta}(X > x) = o(x^{-2})$ is correct but not tight for Gaussian limits. Since the above expression for the tail probability is stated in \citet{Zolotarev} without proof for the $\alpha=1$ case, we include a derivation in Appendix \ref{appTailProbability}.

For the two-component mixture problems considered in this paper, the class of limit distributions that can arise is a proper subset of those listed above. First, the fact that each summand $h(X_i)$ has finite mean implies $\alpha \ge 1$. Second, the fact that the random variables are positive implies maximal skewness with $\beta = 1$.
\begin{equation}\label{eqPAlphaBeta}
P_{\alpha,\beta}(X<0) = \frac{1}{2}\Bigl(1- \frac{b (\alpha-2)}{2}\Bigr),
\end{equation}
where $b$ relates to $\beta$ as \citep[][equation I.19]{Zolotarev}
\[
\beta = \cot\Bigl(\frac{\pi \alpha}{2}\Bigr)\tan \Bigl(\frac{\pi b (\alpha - 2)}{\alpha} \Bigr),
\]
with $\beta = b = 1$ at the boundary of the parameter space for $\beta \in [-1,1]$. It follows that \eqref{eqPAlphaBeta} reduces to $1/\alpha$ for $\beta=1$.

In the asymmetric Cauchy class, the distribution with log characteristic function
\[
\log \psi(t) = -|t| - 2i t \log|t| /\pi
\]
has the maximally skew density,
\begin{equation}\label{skew_cauchy_density}
f(x) = \frac 1 \pi \int_0^\infty e^{-t} \cos(t x + 2t\log t/\pi)\, dt,
\end{equation}
which is shown in Fig.~1. Although the summands in $l'(0)$ have finite mean, the characteristic function for $\alpha=1$
does not have a first-order Taylor expansion so the limit distribution does not have a first moment. The left tail is sub-Cauchy and $P_{1,1}(X < 0) \simeq 0.3652$; the right tail behaviour is $f(x) \sim 2/(\pi x^2)$ or $\bar F(x) \sim 2/(\pi x)$, where $\bar F(x)=P_{1,1}(X>x)$.

\subsection{Stabilizing sequences}
\label{subsec:stabilizing-seq}

A primary difficulty in understanding the boundary behavior of the MLE is finding an appropriate normalization of the sum $\sum_{i=1}^n h(X_i)$ that converges in distribution to a stable law. Since guidance on the construction of the stabilizing sequence is opaque, we provide some background in this section. 

Let $F$ be the cumulative function of a distribution on the positive real line, and let $\bar F(x) = 1 - F(x)$ be the right-tail probability. In order that $F$ belong to the domain of attraction of a stable law with index $0 < \alpha < 2$,
it is necessary and sufficient that the tail be regularly varying, i.e.,
\begin{equation}\label{tail_prob}
\bar F(x) \sim \frac{2C_\alpha} {x^\alpha L(x^\alpha)}
\end{equation}
for large~$x$ and some slowly-varying function $L$.
This is a re-statement of a special case of Theorem~2 from section~35 of \citet{GK1954}. The slow-variation factor in \eqref{tail_prob} is expressed in the form $L(x^\alpha)$ rather than $L(x)$ or $1/L(x)$ as this simplifies the scaling sequence~\eqref{Bn_sequence}. For the moment, the choice of constant is immaterial, but the particular choice
\begin{equation}\label{eqC}
C_\alpha = \pi^{-1} \Gamma(\alpha)\sin(\pi\alpha/2),
\end{equation}
matching the right tail of the limit distribution, will subsequently be convenient.

Let $X_1,\ldots, X_n$ be independent and identically distributed with distribution $F$
in the domain of attraction of some stable law with index~$0 < \alpha \le 2$.
Then there exist deterministic stabilizing sequences $A_n, B_n$ such that
\[
\frac 1 {B_n} \sum_{i=1}^n  X_i - A_n
\]
has the stable limit distribution $P_{\alpha,\beta}$ with index $\alpha$.
The skewness coefficient $|\beta| \le 1$ is a balance between the two tails:
if the support of $F$~is bounded below, then $\beta = 1$.
For $\alpha = 2$, the scaling coefficients are $B_n \propto n^{1/2}$.
Otherwise, for $0 < \alpha < 2$, they are determined by the condition
\begin{equation}\label{eqScalingCondition}
\lim_{n \to \infty} n \bar F(B_n x) = 2C_\alpha x^{-\alpha},
\end{equation}
for each $x > 0$ \citep[][section~35]{GK1954}. Some partial intuition for why equation \eqref{eqScalingCondition} determines the scaling sequence $\{B_n\}$ is that convergence to the stable limit, in terms of the characteristic function $g$ of the original random variables, is equivalent to
\[
e^{-iA_n t} g(t/B_n)^n \rightarrow \psi(t), \quad n\rightarrow \infty.
\]
Since $t/B_n \to 0$ for fixed $t$, it follows that the behaviour of $g$ near zero is key to solving for $B_n$. By the inversion formula for characteristic functions, the behavior of $g$ near the origin manifests in the tail behaviour of $\bar F$. Conditions \eqref{tail_prob} and \eqref{eqScalingCondition} imply
\[
\lim_{n \to \infty} \frac{n B_n^{-\alpha}} {L(B_n^\alpha x^\alpha)} = \lim_{n \to \infty} \frac{n B_n^{-\alpha}} {L(B_n^\alpha)} = 1.
\]
To calculate $B_n$, it suffices to invert the asymptotic relation $\tilde{B}_n L (n\tilde{B}_n) \sim 1$, where $\tilde{B}_n:=B_n^\alpha/n$.
This is achieved through the de Bruijn conjugate $L^\dagger$ \citep{deBruijn}, which is a slowly varying function defined (up to an asymptotic equivalence) by the condition,
\begin{equation}\label{eqDB}
L(x L^\dagger(x)) L^\dagger(x) \sim L^\dagger(x L(x)) L(x) \sim 1,
\end{equation}
as $x \to \infty$.
The condition $\tilde{B}_n L (n\tilde{B}_n)\sim 1$ implies that $\tilde B_n \sim L^\dagger(n)$, which defines the scaling sequence up to asymptotic equivalence as
\begin{equation}\label{Bn_sequence}
B_n \sim (n L^\dagger(n))^{1/\alpha}.
\end{equation}

For a proof of equation \eqref{eqDB}, see Theorem 1.5.13 of \citet{Bingham}. Conjugation is an involution $L^{\dagger\dagger} \sim L$, which is also a group inverse under the compositional operation discussed in Appendix \ref{appDBG}. For the moment, it suffices to remark that if $L(n) \sim K (\log n)^\gamma$ for any real $\gamma$ and $K > 0$, then  $L^\dagger(n) \sim 1/L(n)$. In general, however, the conjugate is not equivalent to the reciprocal, and the conjugate of a functional product is not the product of the conjugates.

For $1 < \alpha \le 2$, the distribution $F$ in \eqref{tail_prob} has a finite mean~$\mu$, 
and the footnote to Theorem~2 in section~35 of \citet{GK1954} gives the centering sequence $A_n \sim n \mu/B_n$, which implies that
\[
\frac 1 {B_n} \sum_{i=1}^n  (X_i - \mu)
\]
has the stable limit distribution $P_{\alpha, 1}$ discussed in the preceding section.
It follows from \eqref{eqPAlphaBeta} that the exceedance event $\bar X_n > \mu$ has a large-sample limit probability
\begin{equation}\label{alpha_limit_prob}
\lim_{n\to\infty} \Prob(\bar X_n > \mu) = P_{\alpha, 1}(\varepsilon > 0) = 1 - 1/\alpha,
\end{equation}
where $\varepsilon$ is a random variable with distribution $P_{\alpha, 1}$. This limit is strictly positive, but not more than one half. The case $0 <\alpha < 1$ does not arise in mixture models and is not discussed here. The single remaining case $\alpha=1$ is a little more complicated because $A_n \neq n\mu/B_n$ even if the mean is finite.
The argument in \S\ref{secIntro} shows that it is also the most important case for signal-detection problems.

\subsection{Stabilizing sequences for Cauchy limits}\label{secCauchy}

The main limit theorems for the boundary-parameter problem considered in this paper are developed here and in \S \ref{L-R_limit}. They may be of standalone interest, and may have other applications.

In order for a distribution $F$ with tail probability $\bar F(x) = 2C_1/(xL(x))$ to have a finite mean, it is necessary and sufficient that the tail contribution be finite.
Integration by parts gives
\begin{eqnarray*}
	\int_T^\infty x \, dF(x) &=& T \bar F(T) + \int_T^\infty \bar F(x)\, dx , \\
	&=& \frac {2C_1} {L(T)} + 2C_1 \int_T^\infty \frac{dx} {x L(x)}.
\end{eqnarray*}
For example, $L(x) = (\log x)^{1+\delta}$ suffices for finiteness
only if $\delta > 0$.
While both terms tend to zero for large~$T$, the second term is dominant.

For $\alpha=1$, the scaling sequence \eqref{Bn_sequence} is $B_n \sim n L^\dagger(n)$. Finiteness of the mean implies $L(n) \to \infty$, and hence $B_n/n \sim L^\dagger(n)\to 0$ as $n\to\infty$.
The centering sequence given in the footnote to Theorem~2 of Gnedenko and Kolmogorov (1954, section~35) is
\begin{equation}\label{eqAn}
A_n \sim n \Im \log \psi(1/B_n) \sim n \int \sin(x/B_n) \, dF(x).
\end{equation}
where $\Im c$ denotes the imaginary part of $c$. See also equation~(6) in Chow and Teugels (1979). We now find an approximation for this sine-integral for a broad class of functions $L(x)$, conveniently parametrized in order to recover several important cases arising in Gaussian mixture models.

\begin{theorem}\label{thmSinIntegral}
	Let $F$ be a finite-mean distribution on the positive real line whose tail is $\bar F(x) \sim 2C_1/(x L(x))$, where $C_1=1/\pi$ and the slowly varying function $L(x)$ admits the parametrization
	\begin{equation}\label{eqLParam}
	L(x) = \left\{ \begin{array}{ll}(\beta_0\log x)^{\delta + 1} e^{(\beta_1\log x)^\gamma} & \beta_1> 0 \\
	(\beta_0\log x)^{\delta + 1} & \beta_1 = 0,
	\end{array}\right.
	\end{equation}
	for $\beta_0>0$. Finiteness of the mean implies either $\beta_1>0$ and $0 < \gamma < 1$,	or $\beta_1 = 0$ and $\delta > 0$. The sine-integral for large $T = 1/|t|$ is,
	\begin{eqnarray*}
		\int_0^\infty \sin (t x)  \, dF(x) 
		&=& \mu t - K_{\delta,\gamma,\beta_1}  \frac{t\,(\log T)^{1-\gamma}} {L(T)} +o\biggl(\frac {t (\log T)^{1-\gamma}} {L(T)} \biggr), \quad |t|\rightarrow 0,
	\end{eqnarray*}
	where $\mu$ is the mean, $K_{\delta,\gamma, \beta_1} = 2C_1/ (\beta_1^{\gamma}\gamma)$ for $\gamma, \beta_1 > 0$, and $K_{\delta,0,0} = 2C_1/\delta$.
\end{theorem}

\begin{remark}\label{remarkMaxSk}
	The form $\bar F(x) = 2C_1/(x L(x))$ implies by \eqref{tail_prob} that $F$ belongs to the domain of attraction of the (maximally skew) Cauchy distribution. The parametrization \eqref{eqLParam} of the slowly varying function $L(x)$ is an encompassing form recovering many important examples as special cases; see \S \ref{secTwoTails}.
\end{remark}

\begin{proof}
	The argument detailed in Appendix \ref{appSinIntegral} shows that the dominant component of the sine-integral for small~$t$ is
	\begin{eqnarray*}
		\int_0^\infty \sin(tx)\, dF(x) &=& t \mu - t \int_T^\infty x \, dF(x)  + O\biggl(\frac t {L(T)} \biggr), \\
		&=& t \mu - 2C_1 t \int_T^\infty \frac{dx}{xL(x)}  + O\biggl(\frac t {L(T)} \biggr),
	\end{eqnarray*}
	where the remainder is of smaller order than the second term.
	
	For $\beta_1 = 0$, the second part is a straightforward integral
	\[
	\int_T^\infty \frac{dx}{xL(x)}=
	\int_T^\infty \frac {dx}{x\, (\beta_0\log x)^{\delta + 1}} = \frac1 {\beta_0\delta\,(\beta_0\log T)^\delta} = 
	\frac{\log T}{\delta L(T)}.
	\]
	For $\beta_1, \gamma > 0$, the transformation $u = (\beta_1\log x)^\gamma$ gives rise to a gamma-tail integral
	\begin{eqnarray*}
		\int_T^\infty \frac {e^{-(\beta_1\log x)^\gamma}\, dx} {x (\beta_0\log x)^{\delta+1} }
		&=&  \frac{1}{\beta_1 \gamma} \frac{\beta_1^{\delta+1}}{\beta_{0}^{\delta+1}} \int_{(\beta_1 \log T)^\gamma}^\infty u^{\delta/\gamma -1} e^{-u} du
		\sim \frac{(\beta_1 \log T)^{1-\gamma}} {\beta_1 \gamma\, L(T)}.
	\end{eqnarray*}
\end{proof}

\begin{corollary}\label{corollAn}
	For a distribution satisfying the conditions of Theorem \ref{thmSinIntegral}, the centering sequence \eqref{eqAn} is 
	\begin{eqnarray*}
		A_n &=& \frac {n \mu} {B_n} - K_{\delta,\gamma,\beta_1}  \frac{n (\log B_n)^{1-\gamma}}  {B_n L(B_n)} (1 + o(1)),\\
		&=& \frac {\mu} {L^\dagger(n)} - K_{\delta,\gamma,\beta_1}  \frac{(\log B_n)^{1-\gamma}} {L^\dagger(n) L(n L^\dagger(n))}(1 + o(1)), \\
		&=& \frac {\mu} {L^\dagger(n)} - K_{\delta,\gamma,\beta_1}(\log n)^{1-\gamma} (1 + o(1)).
	\end{eqnarray*}
	The third line follows from the definition of the de~Bruijn conjugate in \eqref{eqDB}.
\end{corollary}

\begin{corollary}\label{corollXbar}
	For a distribution satisfying the conditions of Theorem \ref{thmSinIntegral}, the distribution of the sample average for large $n$ is
	\[
	\frac{\bar X_n - \mu} {L^\dagger(n)} = - K_{\delta,\gamma,\beta_1}  (\log n)^{1-\gamma} + \varepsilon + o_p(1),
	\]
	where $\varepsilon$ has the Cauchy limit with skewness $\beta=1$.
	Given the right-tail behaviour of the limit distribution,
	\[
	\Prob(\bar X_n > \mu) \sim 2 K_{\delta,\gamma,\beta_1}^{-1} (\log n)^{\gamma-1} / \pi, 
	\]
	which tends to zero at rate $\beta_1^\gamma \gamma (\log n)^{\gamma-1}$ if $\beta_1,\gamma > 0$, or $\delta/\log n$ if $\beta_1=0$.
\end{corollary}

\section{Left boundary behaviour for mixtures}\label{secTwoTails}

Consider the two-component mixture problem $f_\theta = (1-\theta)f_0 + \theta f_1$. Since the event $\hat\theta_n >0$ is equivalent to the event that the sample average of the density ratios $h(X_i)$ exceeds its expectation, the discussion of \S \ref{secLimitDist} has implications for the left boundary behaviour $\PP_0(\hat\theta_n >0)$ in view of equation \eqref{null_tail}. This section explains those implications via several perturbative examples in which the tail of $f_1$ departs from Gaussian in different ways. Example \ref{exampleNearlyNormal} illustrates how different tail properties of $f_1$ induce different limit laws for the sums of the density ratios, and thereby different boundary probabilities. Examples \ref{exampleGaussCauchy2} and \ref{exampleTailIndex} then show that even for $f_1$ inducing a skew-Cauchy domain of attraction, different rates are attained via the expression $L(\eta)=f_1(\xi(\eta))/\xi(\eta)$ in equation \eqref{null_tail}.
The conclusions of Examples \ref{exampleGaussCauchy2} and \ref{exampleTailIndex} are corollaries of Theorem \ref{thmSinIntegral} with the form of $L(\eta)$ in \eqref{eqLParam} deduced from the form of $f_1(\xi(\eta))/\xi(\eta)$.

\begin{example}\label{exampleNearlyNormal}
	Let $\nu > -1$, and suppose that the non-null generator has density
	\[
	f_1(x) \propto |x|^\nu \phi(x),
	\]
	so that $h(x) \propto |x|^{\nu}$ is even.
	For $\nu \neq 0$, both boundary points are extreme.
	
	Ultimate monotonicity requires $\nu > 0$, in which case $h(x) > \eta$ is equivalent to
	$|x| > \hbox{\rm const}\, \eta^{1/\nu}$.
	In that case, the tail behaviour of the density ratio is
	\[
	-\log \Prob_0 (h(X) > \eta) \simeq \hbox{\rm const}\, \eta^{2/\nu} + L(\eta),
	\]
	where $L$ is slowly varying. It follows that all moments exist and $h(X)$ belongs to the domain of attraction of the normal distribution.
	Thus, $\lim \Prob_0(\hat\theta_n > 0) = 1/2$ by the classical theory for boundary parameter problems.
	
	For $-1 < \nu < 0$, the ratio $h(x)\propto |x|^{\nu}$ is monotone decreasing in $|x|$, so condition \eqref{umc} is not satisfied.  However, the tail behaviour can be found by a simpler argument:
	\[
	\Prob_0(h(X) > \eta) = \Prob_0\bigl(|X| < \hbox{\rm const}\, \eta^{1/\nu}\bigr) \sim \hbox{\rm const}\, \eta^{1/\nu}.
	\]
	In other words, the tail is regularly varying with index $\alpha=-1/\nu > 1$.
	In this case, $h(X)$ belongs to the normal domain of attraction only if $\nu \ge -1/2$, in which case $\lim \Prob_0(\hat\theta_n > 0) = 1/2$.  Otherwise, if $-1 < \nu < -1/2$, the behaviour is nonstandard and the limit is $\lim\Prob_0(\hat\theta_n > 0)=1-1/\alpha = 1+\nu$.  \qed
\end{example}

\begin{example}\label{exampleGaussCauchy2}
	Consider a Gaussian mixture with $f_1$ the standard Cauchy density function. The density ratio $h(x) \sim \hbox{\rm const}\,x^{-2} e^{x^2/2}$ is not monotone in $|x|$, but it is monotone for $|x| > 1$, and thus ultimately monotone according to \eqref{umc}. The equation 
	$h(\xi) = \eta$ has an asymptotic solution
    \begin{equation}\label{eqAsympSol}
	\vhalf \xi^2 = \log\eta + o(\log \eta),
	\end{equation}
	so that $\xi^{-1}f_1(\xi) \sim (2\log \eta)^{-3/2}/\pi$, and the tail approximation \eqref{null_tail} is
	\[
	\Prob_0(h(X) > \eta) \sim \frac{2} {\pi \eta\, (2\log\eta)^{3/2}}.
	\]
	This is of the form in Theorem \ref{thmSinIntegral} with $(\beta_0, \beta_1, \delta, \gamma)=(2, 0, 1/2,0)$, leading by Corollary \ref{corollXbar} to the conclusion that 
    \begin{equation}\label{eqApproxGaussCauchy}
	\PP_0(\hat\theta_n>0)\sim \frac{1}{2\log n}. 
	\end{equation}
	\qed
\end{example}

\begin{example}\label{exampleTailIndex}
	Consider a Gaussian mixture in which the non-null generator has  tails satisfying \eqref{tail_index_class} with tail index $0\leq \kappa<1$. The tails are regularly varying for $\kappa=0$ and otherwise exponential. 
    Then $h(\cdot)$ is ultimately monotone, and the equation $h(\xi) = \eta$ has asymptotic solution
	\[
	\xi^2 = 2\log\eta + (\log\eta)^{\kappa} L((2\log\eta)^{1/2}) + o((\log\eta)^{\kappa}).
	\]
	For $\kappa < 1$, the composition $f_1(\xi(\eta))$ is slowly-varying as a function of~$\eta$, and the tail approximation \eqref{null_tail} is
	\[
	\Prob_0(h(X) > \eta)\sim \frac {2 f_1(\xi(\eta))} {\eta\, \xi(\eta)} \sim \frac{2 c_\kappa e^{-(2\log \eta)^\kappa}}{\eta \sqrt{2\log \eta}},
	\]
	where $c_\kappa$ is a normalizing constant. This is of the form in Theorem \ref{thmSinIntegral} with $(\beta_0, \beta_1, \delta, \gamma)=(2/(c_\kappa \pi)^2, 2, -1/2, \kappa)$, leading by Corollary \ref{corollXbar} to the conclusion that 
	\[
	\PP_0(\hat\theta_n>0)\sim \frac{2^\kappa \kappa}{(\log n)^{1-\kappa}}.
	\]
	In particular, for the Laplace distribution $f_1(x)= \frac{1}{2}e^{-|x|}$, $\kappa=1/2$ and the convergence rate is $(\sqrt{2\log n})^{-1}$. \qed
\end{example}

\begin{figure}
	\begin{center}	
        		\includegraphics[trim=0.37in 2.95in 0.85in 3.3in, clip,width=0.49\linewidth]{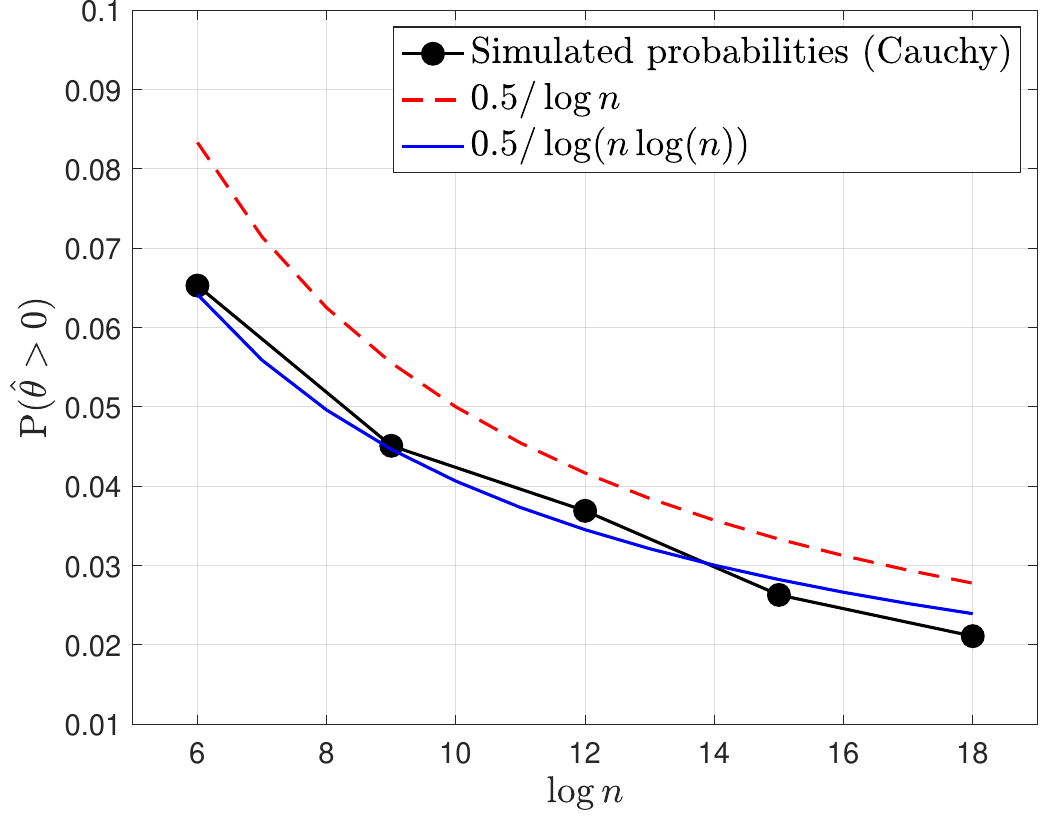}		
		\includegraphics[trim=0.43in 2.95in 0.85in 3.3in, clip,width=0.49\linewidth]{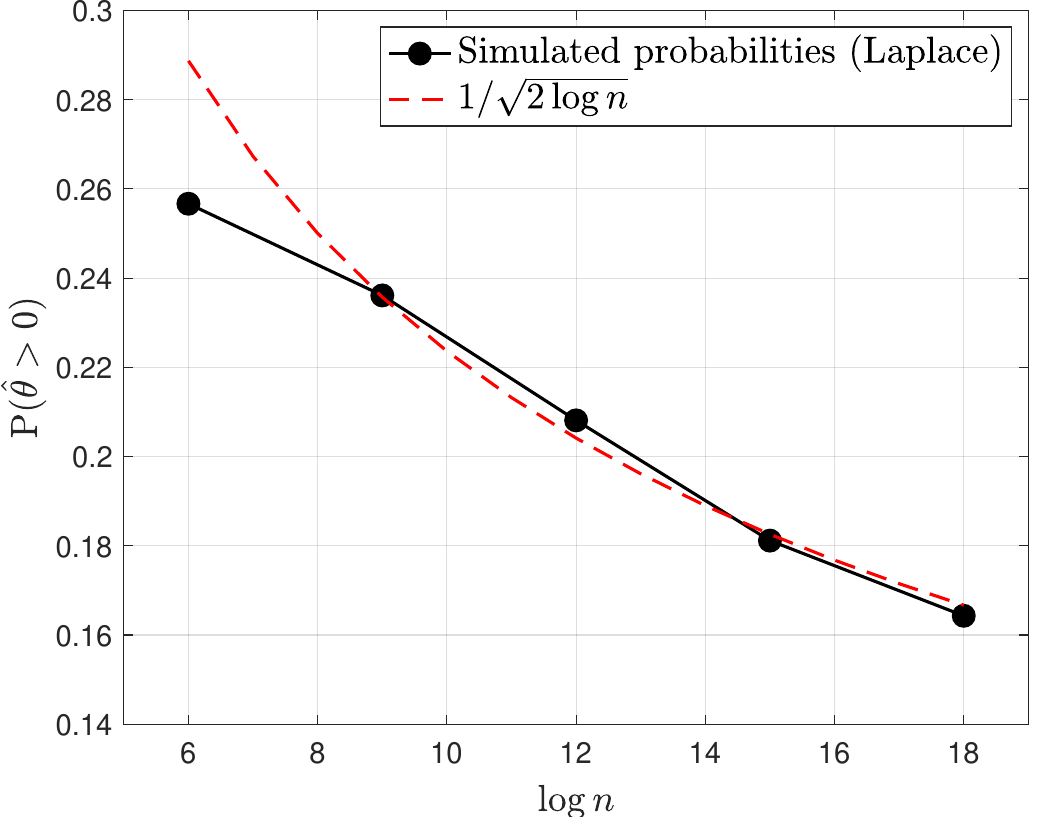}
	\end{center}
	\caption{Simulated probabilities $\Prob(\hat\theta_n>0)$ based on $2\times 10^4$ replicates, where $\hat\theta_n$ estimates $\theta$ in the model $(1-\theta)N(0,1)+\theta F_1$, where $F_1$ is standard Cauchy (left, Fig.~\ref{FigSim}a) and standard Laplace (right, Fig.~\ref{FigSim}b). \label{FigSim}}
\end{figure}

The conclusions of Examples \ref{exampleGaussCauchy2} and \ref{exampleTailIndex} are compared with simulated probabilities in Figure \ref{FigSim}. A referee points out that the accuracy of the approximation in Fig.~\ref{FigSim}a is substantially
inferior to that in Fig.~\ref{FigSim}b. This anomaly may be due to the fact that the default regular-variation approximation
for the Cauchy model
 \[
 \Prob_0(h(X) > \eta) \sim  \frac 2 {\pi \eta\, L(\eta)} = \frac 2 {\pi \eta\, (2 \log \eta)^{3/2}}
 \]
is not entirely satisfactory as an approximation to the tail probability
 \[
 \Prob_0(h(X) > \eta) = \frac 2 {\pi \eta\, (2 \log(\eta\log\eta))^{3/2}} \Bigl( 1 + O((\log\eta)^{-2}) \Bigr)
 \]
 (see Appendix \ref{appRefined}). The natural adjustment $L(\eta) \mapsto L(\eta\log\eta)$ leads to $\Prob(\hat\theta_n > 0) \simeq 1/(2\log(n\log n))$, which is a negligible modification from the vantage of regular variation, but Fig.~\ref{FigSim}a shows that it is an appreciable improvement for typical sample sizes. By contrast, no comparable adjustment is needed for the Laplace tail
 \[
 \Prob_0(h(X) > \eta) = \frac {2C e^{-(2\log\eta)^{1/2}}} {\pi \eta\, (2 \log \eta)^{1/2}} \Bigl( 1 + O((\log\eta)^{-1}) \Bigr). 
 \]

The previous examples illustrate asymptotic type 1 error rates in several instances where the tail of $f_1$ is non-Gaussian. In practice, $f_1$ often arises itself as a signal-plus-Gaussian-noise convolution. If the non-null component $F_1 = N(0,1) \star G$ is a convolution with a symmetric signal distribution $G$, then $F_1$ has tails at least as heavy as those of $G$. We formalize this claim below. 

\begin{proposition}
	\label{prop:tail-inheritance}
Let $0 \le \gamma < 2$, and let $G$ be symmetric with density $g$ satisfying
$-\log g(x) \sim |x|^\gamma L(x)$ for some slowly-varying function~$L$.
Then the Gaussian convolution $F_1 = N(0,1) \star G$ with density $f_1$ also satisfies
$-\log f_1(y) \sim |y|^\gamma L(y)$ for large~$y$. 
\end{proposition}

A proof is given in Appendix \ref{app:tail-prop}. It follows from the above result that if $G$ has tails satisfying \eqref{tail_index_class}, then so does the convolution $F_1=G \star N(0,1)$. Therefore, the distribution of the density ratio under the null remains in the Cauchy domain of attraction, implying that the type 1 error rates derived in Corollary \ref{corollXbar} apply to the case where the non-null generator is a signal-plus-Gaussian noise convolution.

\section{Extreme points and reparametrization}\label{secParam}

\subsection{General insights}

When the data are generated by $\mathbb{P}_\theta$, the left boundary behaviour is non-standard, while at the right boundary, $\theta=1$, the standard asymptotics apply. The distinction between the two cases is that the right boundary can typically be converted to an interior point after reparametrization, while the left boundary point cannot. 

Consider the set of all real $\theta$ in \eqref{eqMixtureModel} for which $f_\theta$ is a valid density function. This is a convex set whose extreme points are
\begin{align}
\begin{split}\label{eq:theta-max}
\theta_{\min} &= \min\{\theta\in\Real : (1-\theta)f_0(\cdot) + \theta f_1(\cdot) \ge 0\} \\
\theta_{\max} &= \max\{\theta\in\Real : (1-\theta)f_0(\cdot) + \theta f_1(\cdot) \ge 0\}.
\end{split}
\end{align}
When $f_1$ has heavier tails than $f_0$, the extreme points typically satisfy $\theta_{\min}=0$ and $\theta_{\max}>~1$, and the original model $\{f_\theta:0\leq \theta\leq 1\}$ is a strict subset of the extended model
\begin{equation}\label{eqExtended}
\conv(f_0, f_{\theta_{\max}}) = \{(1-\rho)f_0 + \rho f_{\theta_{\max}} : 0 \leq \rho \leq 1\}.
\end{equation}
Since $f_{\theta_{\max}}=(1 - \theta_{\max})f_0 + 	\theta_{\max}f_1$, we may express the convex combination $\conv(f_0, f_1)$ in the same parametrization as \eqref{eqExtended}, giving 
\[
\conv(f_0, f_1) =  \{(1-\rho)f_0 + \rho f_{1} : 0 \leq \rho \leq 1/\theta_{\max}\}.
\]
If $\theta_{\max}$ exceeds 1, the right boundary point $\theta=1$ in the original parametrization becomes an interior point in the new parameterization, explaining from a different perspective why the standard asymptotic theory applies in this case. Similar insights were obtained by \citet{PatraSen2016} in a multiple testing context. 

\subsection{Examples}

The following cases typify modelling assumptions used in the empirical Bayes approach to multiple testing and indicate the generality with which the previous conclusions hold.

\begin{example}\label{exampleGaussCauchy1}
	Let $f_0(x)=\phi(x)$ be the standard Gaussian distribution, and let $f_1(x)$ be the standard Cauchy distribution.
	Then $f_\theta$ has a density
	\[
	(1-\theta) \phi(x) + \theta f_1(x) = \phi(x) ( 1 - \theta + \theta h(x) ),
	\]
	where $h(x) = f_1(x)/\phi(x)$ is the density ratio.
	The positivity condition $f_\theta \geq 0$ implies 
	\begin{align*}
	\theta(h(x) - 1 ) \ge -1
	\end{align*}
	for all  $x, \theta$. On the subset for which $h(x) \ge 1$, this implies the lower bound
	\begin{equation}\label{alphamin}
	\theta_{\min} = \sup_{x : h(x) > 1}  \frac{-1} {h(x) - 1}  = 0,
	\end{equation}
	which is zero since $h$ is unbounded as $x\to\infty$.
	On the subset for which $h(x) \le 1$, positivity implies the upper bound
	\begin{equation}\label{alphamax}
	\theta_{\max}  = \inf_{x : h(x) < 1}  \frac{1} {1 - h(x)} =  \frac1 {1 - h(1)} \simeq 2.9218.
	\end{equation}
	The density $f_{\theta_{\max}}$ is illustrated in Figure \ref{fig:trimodal-density}.

\begin{figure}[t]
	\centering
	\includegraphics[width=0.65\textwidth]{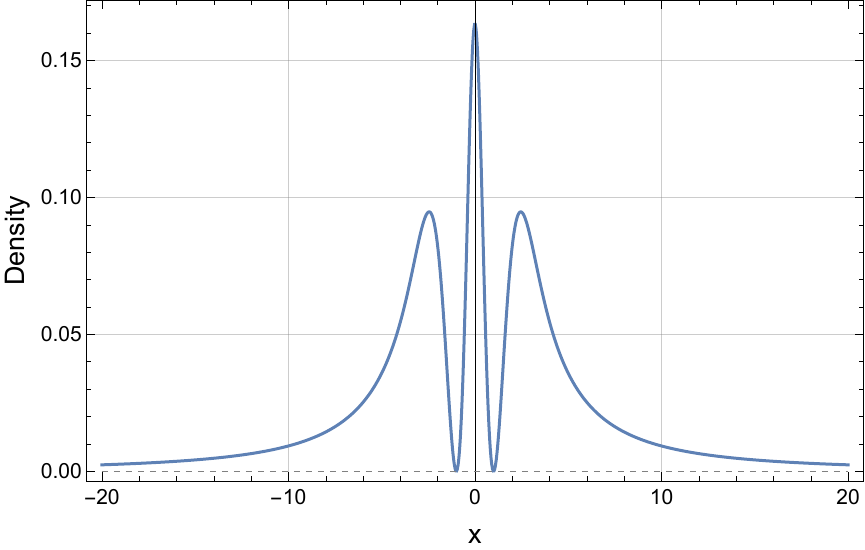}
	\caption{The density function $f_{\theta_{\max}}$ is symmetric and trimodal with zero density at $\pm1$ when $f_1$ is standard Cauchy and $f_0$ is standard normal. It is an extreme point relative to $N(0,1)$ in the sense of \eqref{eq:theta-max}.} 
	\label{fig:trimodal-density}
\end{figure}

As depicted in Figure \ref{fig:trimodal-density}, the level set $\{x \colon h(x) \le 1\}$ is a symmetric interval, approximately  $|x| \le 1.85$,
while the minimum in \eqref{alphamax} occurs at $x=\pm1$, and a local maximum $h(0) \simeq 0.7979$ at $x=0$. This is only partly typical for applied work in certain domains. \qed
\end{example}

The following example illustrates a broad family of symmetric distributions 
for which $h$~is convex and the minimum occurs at the origin, implying $\theta_{\max} = 1/(1 - h(0))$.

\begin{example}\label{exampleConvol}
	Signal-plus-noise model:
	Let $F_0 = N(0,1)$, let $G$ be a symmetric distribution on $\RR$, and let $F_1 = G \star F_0$ be the convolution with density
	\begin{align*}
		f_1(y) &= \phi(y) \int_\RR e^{y x - x^2\!/2}\, G(dx); \\
		h(y) = f_1(y)/\phi(y) &= \int_\RR e^{-x^2\!/2} \cosh(yx)\, G(dx).
	\end{align*}
	Symmetry of $G$ implies that $h(\cdot)$ is a positive combination of $\cosh$-functions,
	and hence that $h(\cdot)$~is symmetric and convex with a minimum at the origin.
	Provided that $G \neq \delta_0$, the minimum is unique, and the value $h(0) = \int e^{-x^2\!/2} \, G(dx)$ is strictly less than one.
	The argument used in Example~\ref{exampleGaussCauchy1} implies $\theta_{\text{min}} = 0$
	and $\theta_{\text{max}} = 1/(1 - h(0)) > 1$ for every symmetric convolution $F_1 = G\star N(0,1)$.
	The upper extremity $F_{\theta_{\text{max}}}$ has zero density at the origin. \qed
\end{example}

\section{Null limit distribution of test statistics}

\subsection{Likelihood-ratio statistic}\label{L-R_limit}
The likelihood-ratio statistic for testing $\theta = 0$ is
\[
\Lambda_n=2\sum_{i=1}^n \log(1+\hat\theta_n(h(X_i)-1)).
\]
Our goal here is to establish the asymptotic null distribution, particularly the conditional distribution given $\hat\theta_n > 0$.

The conventional arguments based on Taylor expansion \citep[see Appendices A.1 and A.3 of][]{Brazzale} do not apply because the $r$th log-likelihood derivative at zero is equal to $(-1)^{r-1}(r-1)!\sum_{i=1}^n(h(X_i)-1)^r$, where $(h(X_i)-1)^r$ belongs to the domain of attraction of the stable law with index $1/r$. The first derivative does not have a finite variance, so there is no concept of Fisher information.  Higher-order derivatives do not have an expectation, and are strongly dependent. 
Nevertheless, one series of simulations shown in the right panel of Figure 4 suggests that the asymptotic null distribution given $\hat\theta_n > 0$ is close to $\kappa_n \chi_1^2$, where  $\kappa_n \to 1$ tends to one. This section offers an explanation and a derivation of the correct limit distribution,
which is not~$\chi_1^2$ in the setting of Theorem~\ref{thmSinIntegral}.

In order to prove Theorem \ref{LR_limit} we require the following theorem concerning the joint distribution of a sample average $\bar Z_n$ and the maximum order statistic $\Zmax$, which is of independent interest. The switch in notation from $X$ representing a generic random variable in \S \ref{secLimitDist} to $Z$ here is because $X$ in the present section is used to represent a different object to that for which the theorem is to be applied. The joint limiting distribution arises in the calculation of the conditional limiting distribution of the likelihood ratio statistic conditional on positivity, which is shown in the proof of Theorem \ref{LR_limit} to have a dominant expression in terms of the top order statistic.

\begin{theorem}\label{cjl_theorem}
	Let $F$ be a zero-mean distribution with support $(-1, \infty)$ and tail satisfying the conditions of Theorem~\ref{thmSinIntegral}, and let $Z_1,\ldots, Z_n$ be an iid sample from $F$. Let $T_n = B_n|A_n| = K_{\delta,\gamma,\beta_1} B_n (\log n)^{1-\gamma}$, where $B_n=n/L(n)$ and $A_n$ is given in Corollary~\ref{corollAn}.
	Given $\bar Z_n > 0$, the conditional limit distribution of $(\bar Z_n, \Zmax)$ is such that
	\[
	\biggl(\frac{n\bar Z_n} {T_n},\, \frac{\Zmax} {T_n} \biggr) \sim \biggl(\frac U{1-U},\, \frac 1 {1-U} \biggr),
	\]
	where $U$ is uniform on $(0, 1)$.
	It follows that $n\bar Z_n/T_n \sim T_n / (n\bar Z_n)$ is conditionally self-reciprocal,
	the ratio $n\bar Z_n / \Zmax \sim U$ is conditionally uniform, 
	and the support of the joint distribution degenerates to the line $\Zmax/T_n - n\bar Z_n/T_n -1  \to 0$.
\end{theorem}

\begin{remark}
	The unconditional scaling factor for both $n\bar Z_n$ and $\Zmax$ is $B_n$, so
	$(n\bar Z_n/B_n - A_n, \Zmax/B_n)$ has a joint limit distribution, which has been studied by
	Chow and Teugels (1979) in this setting.
	However, the conditioning event $\bar Z_n > 0$ has probability $O(B_n/T_n)$ tending to zero, so
	the joint limit distribution does not determine the limit of conditional distributions.
	Theorem~\ref{cjl_theorem} implies that $n\bar Z_n > \epsilon T_n$ and $\Zmax > (1+\epsilon)T_n$ are asymptotically equivalent events for each $\epsilon > 0$.
\end{remark}

To prove the theorem, we first state a technical lemma that relates the sample mean to the sample maximum in the event where the latter is unusually large.

\begin{lemma}
	\label{lem:sum-max-dependence}
	Let $F$ be a zero-mean distribution with support $(-1, \infty)$ and tail satisfying the conditions of Theorem~\ref{thmSinIntegral}, and let $Z_1,\ldots, Z_n$ be an iid sample from $F$. Let $B_n,T_n$ be defined as in Theorem \ref{cjl_theorem}. For any fixed $y>0$, conditional on $Z_{(n)}>y T_n$, we have
	\begin{align}
	\label{eq:sum-max-dependence}
	\frac{n\bar{Z}_{n}}{T_n} = \frac{Z_{(n)}+\sum_{k\geq 1}Z_{(n-k)}}{T_n} = \frac{Z_{(n)}}{T_n}-1+o_p(1).
	\end{align}
	Consequently, for any fixed $y>1$
	\begin{align*}
	\lim_{n\to\infty} \; &\Prob\left( \bar{Z}_n>0 \mid Z_{(n)}>yT_n \right) = 1,
	\end{align*}
	so that $\Prob(\bar{Z}_n>0, Z_{(n)}>yT_n) \sim \Prob(Z_{(n)}>yT_n)$ as $n\to\infty$.
\end{lemma}

\begin{proof}
	Let $Z_{(n-k)}$ be the $k$th order statistic with $Z_{(n+1)} = \infty$.
	Since the transformed variables $F(Z_1), \ldots, F(Z_n)$ are independent uniform, the joint distribution of the top order statistics is such that successive differences of the probability-integral transformed values have exactly the same distribution as successive spacings of the top~$k+1$ of $n$ independent uniform order statistics, i.e.
	\[
	n \{\bar F(Z_{(n-r)}) - \bar F(Z_{(n-r+1)}) \}_{0\le r \le k}
	\; \stackrel{(d)}{\to} \, (e_0, \ldots, e_k),
	\]
     where $e_0, \ldots, e_k$ are independent unit exponential variables, whose distribution approximates that of the spacings among the top $k+1$ uniform order statistics, for large $n$ and $k\ll n$. Since $u:=~\bar F(z) \sim  2C_1/(z L(z))$ for large $z$ implies $zL(z)\sim 2C_1/u$, and since the property $L^\dagger(z L (z))L(z) \sim 1$ implies $L(z)\sim 1/L^\dagger(2C_1/u)$, the inverse function inherits the asymptotic behaviour
	\[
	\bar{F}^{-1}(u)\sim (2C_1/u)\; L^\dagger(2C_1/u), \quad  (u \rightarrow 0).
	\] 
	Now since $n\bar{F}(Z_{(n)}) \stackrel{(d)}{\to} e_0$ and 
	\[
	\bar F^{-1}(e_0/n) \;\sim (2C_1n/e_0) \;  L^\dagger(2C_1n/e_0)  \sim \frac{2C_1}{e_0} \; n L^\dagger(n) \sim 2C_1 B_n/e_0,
	\]
	it follows that $Z_{(n)}/B_n \stackrel{(d)}{\to}2C_1/e_0$. 
	Similarly, since $n\bar{F}(Z_{(n-k)}) \stackrel{(d)}{\to} e_0+\dots+e_k$, we obtain
	\begin{equation}\label{e_rep}
	\bar{F}^{-1}((e_0+\dots+e_k)/n) \sim \frac{2C_1 B_n}{e_1+\dots+e_k}
	\end{equation}
	which implies $Z_{(n-k)}/B_n \stackrel{(d)}{\to} 2C_1/(e_0+\dots+e_k)$.

	If $\Zmax > yT_n$ for some $y > 0$, then $e_0$~is unusually small compared with $e_1, e_2,\ldots$, so that $e_0 + \cdots + e_k \sim e_1 + \cdots + e_k$. This is because $e_1,\dots,e_k$ are unit exponential variables that are independent of $e_0$, and on the event $\{Z_{(n)}>yT_n\}$, we have
    \begin{align*}
        e_0 \sim n\bar{F}(Z_{(n)}) \leq n\bar{F}(yT_n) \sim \frac{2nC_1}{yT_n L(yT_n)} = \frac{2C_1 L(n)}{K_{\delta,\gamma,\beta_1} yL(yT_n)(\log n)^{1-\gamma}}
    \end{align*}
    by definition of $T_n=K_{\delta,\gamma,\beta_1} n (\log n)^{1-\gamma}/L(n)$. Since $L$ is slowly varying, $\frac{L(yT_n)}{L(n)} \sim 1$ as $n\to\infty$, so the right hand side of the above tends to zero.
    It follows that the conditional distribution of $Z_{(n-1)}, Z_{(n-2)},\ldots$ given $\Zmax > y T_n$ is approximately the same
	as the unconditional distribution of $\Zmax, Z_{(n-1)},\dots$.
	Since the top order statistics $Z_{(n-1)}, Z_{(n-2)},\ldots$ dominate the partial sum $n \bar Z_n - \Zmax$, on the event $\Zmax > y T_n$, Corollary \ref{corollXbar} applied with $\sum_{j \ge1} Z_{(n-j)}$ in place of $n \bar X_n$ implies
	\begin{equation}\label{eq16}
	n \bar Z_n = \Zmax + \sum_{j \ge1} Z_{(n-j)} \sim \Zmax - T_n + O_p(B_n),
	\end{equation}
	which implies \eqref{eq:sum-max-dependence}. For $y>1$, $B_n = o(T_n)$ implies
	\begin{align*}
	\Prob(n\bar{Z}_n > 0 \mid Z_{(n)}>yT_n) \to 1.
	\end{align*}
\end{proof}

\begin{proof}[Proof of Theorem \ref{cjl_theorem}]
	For $U \sim \text{Uniform}(0,1)$,
	\begin{align*}
	\Prob\left( \frac{U}{1-U}>z ,\, \frac{1}{1-U}>y \right) = \min\left\{\frac{1}{1+z} , \frac{1}{y} \right\},\hspace{2em} y> 1, z> 0.
	\end{align*}
	By the last line of Lemma \ref{lem:sum-max-dependence}, we have
	\begin{align*}
	\Prob\left( \frac{n\bar{Z}_n}{T_n}>z,\frac{Z_{(n)}}{T_n}>y \mathrel{\Big |} \bar{Z}_n>0\right) &\sim \Prob\left( \frac{n\bar{Z}_n}{T_n}>z \mathrel{\Big |} \frac{Z_{(n)}}{T_n}>y, \bar{Z}_n>0 \right) \frac{\Prob\left(Z_{(n)}/T_n>y\right)}{\Prob(\bar{Z}_n>0)}.
	\end{align*}
	For any fixed $y$, the tail property of $\bar{F}$ implies
	\[
	\Prob (Z_{(n)} > yT_n) \sim n\bar{F}(yT_n) \sim \frac{2C_1 B_n}{y T_n} ,
	\]
	and Corollary \ref{corollXbar} implies $
	\Prob(\bar{Z}_n>0) \sim 2C_1 B_n/T_n$. It follows that for any fixed $y$,
	\begin{align}
	\label{eq:prob-ratio}
	\lim_{n \to \infty} \frac{\Prob(Z_{(n)}/T_n>y)}{\Prob(\bar{Z}_n>0)} \to \frac{1}{y}.
	\end{align}
	
	Since $Z_{(n)}/T_n>y$ for $y>1$ implies $\bar{Z}_n>0$ with probability 1 in the limit, we have
	\begin{align*}
	\Prob\left( \frac{n\bar{Z}_n}{T_n}>z \mathrel{\Big |} \frac{Z_{(n)}}{T_n}>y, \bar{Z}_n>0 \right) &\sim \Prob\left( \frac{n\bar{Z}_n}{T_n}>z \mathrel{\Big |} \frac{Z_{(n)}}{T_n}>y \right) .
	\end{align*}
	Lemma \ref{lem:sum-max-dependence} implies that, conditional on $Z_{(n)}>y T_n$, 
	\begin{align*}
	\frac{n\bar{Z}_{n}}{T_n} = \frac{Z_{(n)}+\sum_{k\geq 2}Z_{(n-k+1)}}{T_n} \sim \frac{Z_{(n)}}{T_n}-1+o_p(1),
	\end{align*}
	so the previous expression is, by the argument leading to \eqref{eq:prob-ratio},
	\begin{align*}
	\Prob\left( \frac{n\bar{Z}_n}{T_n}>z \mathrel{\Big |} \frac{Z_{(n)}}{T_n}>y \right) &\sim \Prob\left( \frac{Z_{(n)}}{T_n}>1+z \mathrel{\Big |} \frac{Z_{(n)}}{T_n}>y \right) \to \begin{cases}
	\frac{y}{1+z} \hspace{1em} &\text{if }1+z>y \\
	1 &\text{otherwise}.
	\end{cases}
	\end{align*}
	Together with \eqref{eq:prob-ratio}, this proves the claim.
\end{proof}

\begin{theorem}\label{LR_limit}
	Let $X_1,\ldots, X_n$ be independent standard Gaussian,
	let $h(X_i) \ge 0$~be in the domain of attraction of the maximally skew Cauchy law, and
	let $G(\cdot)$ be the cumulative distribution function whose $u$th quantile is
	\begin{equation}\label{u-quantile}
	G^{-1}(u) = -2u - 2\log(1-u) = 2\sum_{r \ge 2} u^r/r.
	\end{equation}
	Then, the conditional limit distribution of the likelihood-ratio statistic is
	\[
	\lim_{n\to \infty} \Prob_0(\Lambda_n \le x \given \Lambda_n > 0) = G(x).
	\]
\end{theorem}

\begin{remark}
	The limit distribution is not dissimilar to $\chi_1^2$.  
	Both densities behave like $x^{-1/2}$ near the origin. The first four cumulants of $\chi_1^2$ are 1, 2, 8, and 96, while those of $G$ are 1, 7/3, 32/3, and 3194/45. However, for $X \sim G$, the log density of $X^{1/2}$ has a Taylor expansion
	\[
	\log\bigl(2 x g(x^2)\bigr) = -\frac{2x} 3 - \frac{5x^2}{36} - \frac{23 x^3}{810} - \frac{31 x^4}{6480} + O(x^5),
	\]
	which is essentially linear for $x < 1$, while the corresponding half-normal log density $\log(2/\pi)/2 - x^2/2$ is exactly quadratic and negative at zero.
	The difference between $2xg(x^2)$ and the degree-4 Taylor approximation is less than 1\% for $x < 3$.
\end{remark}

\begin{remark}
	As discussed in \S \ref{secStable}, the maximally-skew domain of attraction is the only one that is relevant for the two-component mixture problems we have in mind, and the proof is specific to that case through its dependence on Theorem \ref{cjl_theorem} and via Remark \ref{remarkMaxSk}.
\end{remark}

\begin{proof}
	For notational simplicity  we set $Z = h(X) - 1 \sim F$, so that $F$ is in the maximally-skew Cauchy domain of attraction with zero mean, and Theorem \ref{cjl_theorem} applies. Given $\bar Z_n > 0$, we construct an approximation for $l(\cdot)$ and its derivatives $l^{(r)}(0) = (-1)^{r-1} (r-1)! S_r$, where
	\[
	S_r = \sum_{i=1}^n Z_i^r =  \sum_{k\ge0} Z_{(n-k)}^r = Z_{(n)}^r + \sum_{k\ge1} Z_{(n-k)}^r.
	\]
	Both $n\bar Z_n$ and $Z_{(n)}$ are conditionally $O_p(T_n)$, while the conditional distribution of the remaining order statistics $Z_{(n-1)}, Z_{(n-2)},\ldots$ is asymptotically the same as the unconditional distribution of $Z_{(n)}, Z_{(n-1)},\ldots$.
	Since $Z_i^r$ is in the domain of attraction of the stable law with index $1/r$, the scaling constant is $B_n^r$, and we have
	\[
	S_r - Z_{(n)}^r =  \sum_{k\ge1} Z_{(n-k)}^r = \left\{
	\begin{array}{ll} n\bar Z_n - Z_{(n)} &\quad r=1, \\
	O_p(B_n^r) & \quad r > 1. \end{array} \right.
	\]
	It follows that $S_r = Z_{(n)}^r(1 + o_p(1))$ for $r \ge 2$. 
	
	Consider the log-likelihood function 
	\[
	\ell(\theta) = \sum_{i=1}^n \log\{(1-\theta)f_0(X_i) + \theta f_1(X_i)\},
	\]
	whose first two derivatives are given in \eqref{eqDeriv}. A first-order Taylor expansion around zero gives the local approximation
	\begin{equation}\label{eqLLL}
	\tilde l(\theta) = \log(1 + \theta Z_{(n)}) + \theta (S_1 - Z_{(n)}),
	\end{equation}
	where, in the Taylor expansion, we have substituted $Z_{(n)}$ for $S_1$ in the expression for $l'(0)$, and added and subtracted $\log(1+\theta Z_{(n)})=\theta Z_{(n)}+O(\theta^2)$. The approximation satisfies $l(\theta) - \tilde l(\theta) = o_p(1)$ for $\theta = O(T_n^{-1})$ owing to the form of the higher derivatives at zero. 
	
	Theorem \ref{cjl_theorem} shows that $R = S_1/Z_{(n)}$ is less than one with high probability for large~$n$, in which case $\tilde l$ has a maximum at $\tilde\theta_n Z_{(n)} = R/(1-R)$.  In that case $R = \tilde\theta Z_{(n)}/(1 + \tilde\theta Z_{(n)})$, and the approximate likelihood-ratio statistic
	\begin{equation}\label{llr_limit}
	2\tilde l(\tilde\theta) = -2R - 2\log(1 - R) = 2\sum_{k \ge 2} R^k/k = \Lambda_n + o_p(1)
	\end{equation}
	is a monotone function of~$R$.
	Theorem \ref{cjl_theorem} also shows that the limit distribution of $R$ given $\bar Z_n > 0$ is uniform on $(0, 1)$, implying that $\Lambda_n \sim G$ in the limit.
\end{proof}

Figure~4 shows a histogram of $R$ and the Bartlett-adjusted likelihood ratio statistic $(\Lambda_n/\hat\kappa_n)^{1/2}$ using the exact likelihood and the exact maximum in the Gauss-Cauchy mixture model restricted to 2000 out of 74058 samples for which $\bar Z_n > 0$.
For this model, $L(x) = (2\log x)^{3/2}$ is the slow-variation function, 
$B_n \sim  n /(2\log n)^{3/2}$ and $T_n \sim 4\pi^{-1} B_n \log n$.
For the simulation, $n = 10^7$,
$\hat\kappa_n = 1.0170$ is the sample average,
and $R > 1$ in 31 cases.
The sample cumulant ratios $k_2/2$ and $k_3/8$ for $\Lambda_n/\hat\kappa_n$ compared with $\chi_1^2$ are 1.160 and 1.346, which are surprisingly close to the theoretical limit values $7/6$ and $4/3$ respectively.
The new limit density is shown on the same square-root scale, together with the half-normal density for comparison.
While the difference between the two distributions is not large, the histogram clearly favours~$G$.
A standard 20-bin $\chi^2$-test gives $X^2=40.02$ for the $\chi_1^2$ distribution, and $X^2 = 15.22$ for~$G$.

\begin{figure}
	\begin{center}
		\includegraphics[trim=0.5in 2.95in 0.8in 3.2in, clip,width=0.49\linewidth]{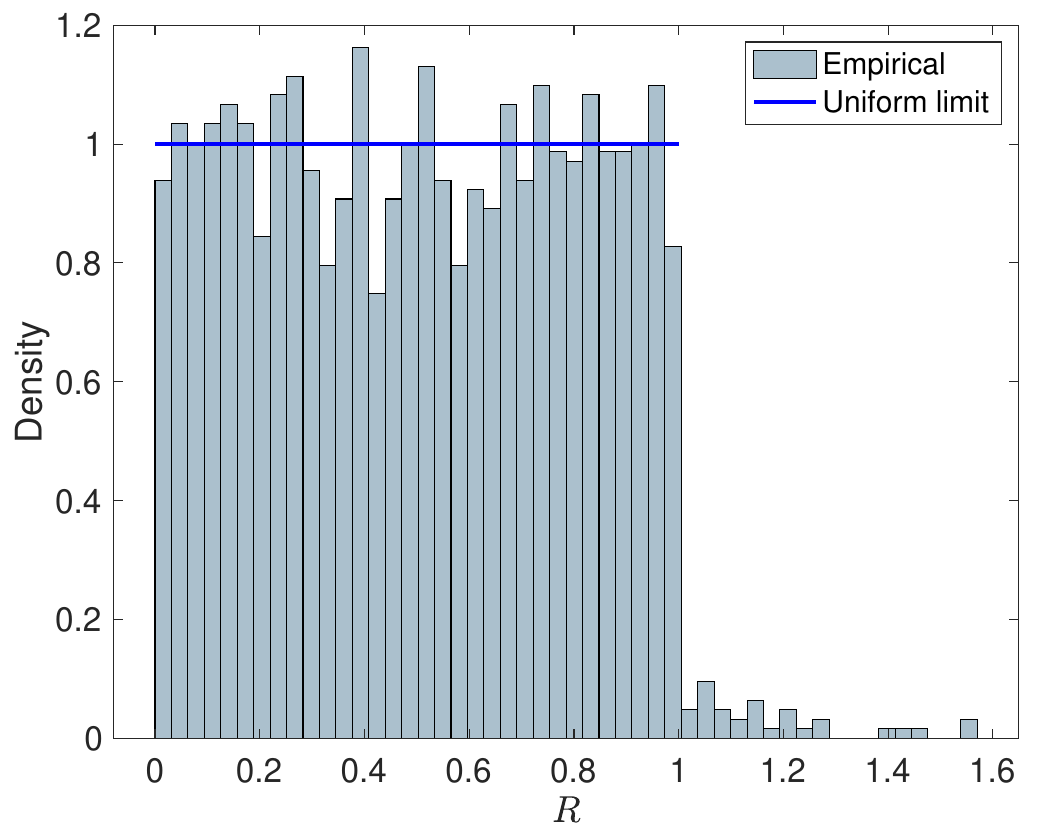}
		\includegraphics[trim=0.5in 2.95in 0.8in 3.2in, clip,width=0.49\linewidth]{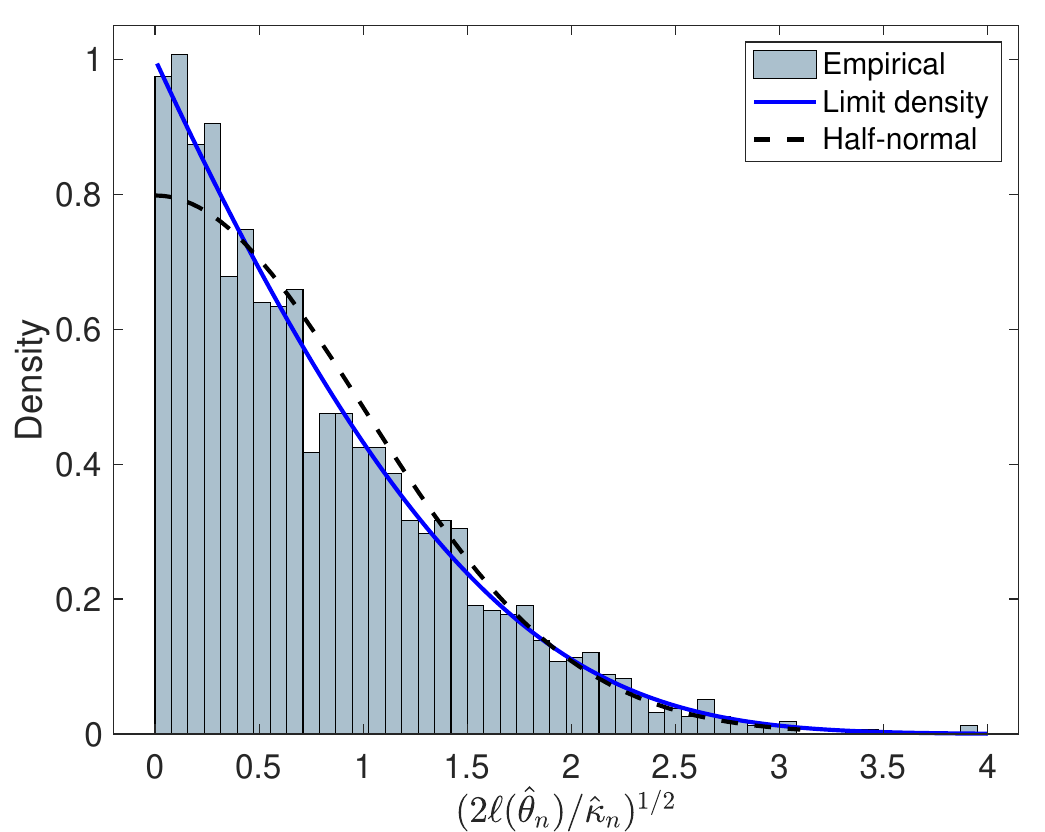}
	\end{center}
	\caption{Histogram of $R = n\bar Z_n/Z_{(n)}$ (left) and $(2l(\hat\theta_n)/\hat\kappa_n)^{1/2}$ (right) for two thousand simulations of the
		Gauss-Cauchy mixture model with $n=10^7$ observations restricted to samples for which $\bar Z_n > 0$.}
\end{figure}

Analogous simulations for the Gauss-Laplace mixture give virtually identical results for the likelihood-ratio statistic, though with $\hat\kappa_n = 0.9566$.
However, the fraction of $R$-values greater than one was 249/2000, or 12\%, consistent with a $(\log n)^{-1/2}$ convergence rate.
The argument used in the proof of Theorem~\ref{LR_limit} suggests that the rate of
convergence of $\Lambda_n$ to~$G$ might be $(\log n)^{-1/2}$ or perhaps $(\log n)^{-1}$, but the simulations suggest a faster rate, particularly after Bartlett correction.

The local approximation \eqref{eqLLL} and ensuing argument is valid for any iid problem with a boundary at zero provided that the log likelihood derivative is a sum of zero mean random variables in the Cauchy domain of attraction. There may therefore be other statistical problems besides the two-component mixture problem in which the conditional limit distribution \eqref{u-quantile} arises; the only difficulty in establishing that $l(\theta) - \tilde l(\theta) = o_p(1)$, as the adequacy of the approximation \eqref{eqLLL} depends on the behaviour of higher derivatives, which is context specific.  

\subsection{Wald and Rao statistics}

The likelihood ratio test discussed in \S \ref{L-R_limit} one of three classicial approaches to testing $\theta=0$. The conventional Wald and Rao statistics for testing the same hypothesis are
\[
\hat\theta (i(\hat\theta))^{1/2}, \quad \hbox{and}\quad \frac{l'(0)} {i(0)^{1/2}},
\]
where $l(\theta)$ is the log likelihood function and $i(\theta)$ is the Fisher information. Technically speaking, the Rao statistic does not exist in the mixture setting because $i(0)$ is not finite.
For present purposes, however, we substitute $-l''(\theta)$ for $i(\theta)$ in both.

Given $\hat\theta > 0$, the approximate likelihood in section~4.3 implies
\begin{eqnarray*}
	{\tilde l}'(\theta) &=& \frac{Z_{(n)}} {1 + \theta Z_{(n)}} + n\bar Z_n - Z_{(n)}, \\
	Z_{(n)}\tilde\theta &=& R/(1-R), \\
	-{\tilde l}''(0) &=& Z_{(n)}^2, \\
	-{\tilde l}''(\tilde\theta) &=& Z_{(n)}^2(1 - R)^2 = (Z_{(n)} - n \bar Z_n)^2 ,
\end{eqnarray*}
where $R = n\bar Z_n/Z_{(n)}$.
It follows that the modified Wald and Rao statistics are
\begin{eqnarray*}
	\tilde\theta (-{\tilde l}''(\tilde\theta))^{1/2} &=& \frac R {(1-R) Z_{(n)}} Z_{(n)} (1-R) = R \\
	\frac{{\tilde l}'(0)} {(-{\tilde l}''(0))^{1/2}} &=& \frac {n \bar Z_n} {Z_{(n)}} = R.
\end{eqnarray*}
In other words, the two statistics are conditionally equivalent given $\hat\theta > 0$, both uniformly distributed on $(0, 1)$.
The squared statistics are not approximately the same as the likelihood-ratio statistic,
although all three are asymptotically equivalent in the sense that each is a monotone function of~$R$.  Conventional standard normal approximations are incorrect.

\section{Nonparametric tail classes}\label{secComposite}

\subsection{Composite mixtures}

In the empirical Bayes approach to multiple testing, a test statistic is constructed at each site, whose distribution is Gaussian if the site in question contributes only noise, thus treating $f_0$ as standard Gaussian is natural. In the presence of signal, however, no parametric distribution for the test statistic is guaranteed, thus it is common to leave it unspecified but from a distribution with tails heavier than Gaussian.

A composite Gaussian mixture is a  family of distributions
\[
\{F_\theta = (1-\theta)F_0 + \theta F_1 : 0 \le\theta\le 1,\, F_1 \in \F\}
\]
in which $F_0 = N(0,1)$.
The non-null distributions belong to some specified non-empty family~$\F$, which need not be convex.
Models of this type have been considered by Efron et al.~(2001), Patra and Sen (2016). In the modern style, these are invariably nonparametric in the counter-semantic sense that $\F$~is not a finite-dimensional manifold.

The emphasis in this section is on composite mixtures arising in the signal-detection setting where each distribution is
necessarily a Gaussian convolution $F_\theta = N(0, 1) \star P$ with some symmetric signal distribution~$P$.
This condition is sufficient to ensure that each $F_\theta$ has a bounded continuous density $f_\theta(x) \le \phi(0)$, and hence that
each distribution in~$\F$ also has a bounded continuous density.
Also, under suitable conditions (Proposition~\ref{prop:tail-inheritance}), the tails of $F_\theta$ are similar to those of~$P$.

\subsection{Equivalence classes}\label{equiv_class}

Let $F_1, F_2$ be two symmetric distributions having bounded continuous densities such that the  ratio has a finite non-zero limit:
\begin{equation}\label{eq_tail_density}
\lim_{x \to\infty} \frac{f_1(x)} {f_2(x)} = K.
\end{equation}
Then $h_1(x) \sim K h_2(x)$, where $h_r(x) = f_r(x)/\phi(x)$ is the density ratio arising in the elementary mixture model
$(1-\theta) \phi(x) + \theta f_r(x)$.
On the assumption that $h_1(X)$ (and hence also $h_2(X)$), belongs to the Cauchy domain of attraction, it follows that
\begin{eqnarray*}
	\frac 1 {B_{1,n}} \sum \bigl(h_1(X_i) - 1\bigr) &=& A_{1,n} + \varepsilon_1 + o_p(1), \\
	\frac 1 {B_{2,n}} \sum \bigl(h_2(X_i) - 1\bigr) &=& A_{2,n} + \varepsilon_2 + o_p(1),
\end{eqnarray*}
where the stabilizing sequences satisfy
\[
B_{1,n} \sim K B_{2,n}, \quad A_{2,n} \sim A_{1,n}, \quad T_{1,n}  \sim K T_{2,n}.
\]
If $\hat\theta_{1,n} > 0$, with high $\Prob_0$-probability $h_1(\Xmax) > T_{1,n}$ implying $h_2(\Xmax) > T_{2,n}$, and hence $\hat\theta_{2,n} > 0$ with high probability.
Moreover, given $\hat\theta_{1,n} > 0$, the ratios satisfy
\[
\frac 1 {h_1(\Xmax)} \sum \bigl(h_1(X_i) - 1\bigr) = \frac 1 {h_2(\Xmax)} \sum \bigl(h_2(X_i) - 1\bigr) + o_p(1).
\]
It follows that the likelihood-ratio statistics $\Lambda_{1,n}, \Lambda_{2,n}$ are asymptotically equal:
\[
\lim_{n\to\infty} \Prob_0\Bigl(|\Lambda_{1,n} - \Lambda_{2,n}| > \epsilon \mathrel{\Big |} \hat\theta_n > 0 \Bigr) = 0
\]
for every $\epsilon > 0$.
For example, if $f_1, f_2$ are Student~$t$ distributions on $\nu$ degrees of freedom with scale parameters $\sigma_1, \sigma_2$, then \eqref{eq_tail_density} is satisfied with $K = (\sigma_1/\sigma_2)^\nu$, implying that the likelihood ratio statistics are asymptotically equal given $\hat\theta_n > 0$.

\subsection{Implications}\label{secImplications}
%First implication: 
The first implication is as follows.
Let $\E(F_1)$ be the set of distributions that are tail-equivalent to $F_1$
in the sense of~\eqref{eq_tail_density},
let $\Lambda_n(F_1)$ be the likelihood-ratio statistic in the elementary mixture $\conv(F_0, F_1)$,
and let $\Lambda_n(\hat F)$ be the maximized likelihood-ratio statistic over $\theta$ and over the class $\mathcal{E}(F_1)$. Given that $\hat\theta_n > 0$, the difference $\Lambda_n(F) - \Lambda_n(F_1)$ is $o_p(1)$ for an arbitrary $F$ in $\mathcal{E}(F_1)$. This implies that $\Lambda_n(\hat F) - \Lambda_n(F_1)$ is $o_p(1)$ so that
the conditional distribution of the maximized likelihood-ratio statistic is
\[
\lim_{n \to \infty} \Prob_0(\Lambda_n(\hat F) < x \mid \hat\theta_n >0) = G(x).
\]
From the perspective of statistical modelling and testing, no advantage is gained by extending the single distribution $F_1$
to the nonparametric composite mixture with $\F = \E(F_1)$.

The second implication is that the conditional limit distribution of the maximized likelihood-ratio statistic
$\Lambda_n(\hat F)$ given $\hat\theta_n > 0$, may depend on the topology of the quotient space $\F/\E$, i.e.,~the space of equivalence classes.
In particular, if $\F/\E$ is a finite-dimensional manifold of dimension~$d$, the limit distribution $G_d$ may depend on~$d$.  Otherwise, $\F$~plays no role in the limit distribution.

There are some parallels between tail-equivalence as defined by \eqref{eq_tail_density}, and sparse-equivalence as defined by \citet{McCP}. Signal distributions $P$ and $P'$ within the same sparse-equivalence class are statistically indistinguishable based on observations from the convolutions $P\star  N(0,1)$ and $P' \star N(0,1)$, and have the same sparsity implications in the relevant statistical sense. Formally, two signal distributions $P, P'$ are termed sparse-equivalent if their normalized exceedance measures are equal: $H(dx) = H'(dx)$. Two sparse signals having exceedance measures $H\neq H'$ give rise to mixtures that are tail-equivalent in the sense of \eqref{eq_tail_density} if the tail ratio  has a finite non-zero  limit $H(x^+) /H'(x^+) \sim K$.  Thus sparse equivalence is finer than tail equivalence. Sparse equivalence is relevant for identifiability of mixtures generated by Gaussian convolution with a sparse signal; tail equivalence is relevant in the present setting where the focus is exclusively on the boundary behaviour of the likelihood-ratio and related statistics.

\subsection{Example}\label{secExample}
For $0 < \nu < 2$, let $\zeta_\nu(x)$ be the inverse-power zeta-function defined in
McCullagh and Polson (2018) and expressed as a convergent power series
\[
\zeta_\nu(x) = \frac{\nu(2-\nu)} {\Gamma(2 - \nu/2)} \sum_{r=1}^\infty \frac{2^{r-2} \Gamma(r - \nu/2)\, x^{2r}} {(2r)!}.
\]
The product $\psi_\nu(x) = \phi(x) \zeta_\nu(x)$ is the density of a symmetric bimodal distribution $\Psi_\nu$ whose tails are regularly-varying with index~$-\nu$.
For purposes of likelihood maximization, it is convenient to extend the model by continuity to the upper boundary: $\zeta_2(x) = \lim_{\nu\to 2} \zeta_\nu(x) = x^2$ and $\psi_2(x) = x^2 \phi(x)$.

One crucial feature of this family is that $\psi_\nu(X)$ belongs to the skew-Cauchy domain of attraction in all cases except $\nu=2$, which belongs to the Gaussian domain.  In the non-null sparse-signal-plus-noise setting, the maximum of the likelihood tends to occur near the upper boundary $\nu = 2$ only if the signal has finite variance; otherwise, if the variance is not finite, the maximum occurs at an interior point with high probability.  To understand this phenomenon, observe that if $Q_\sigma$ is symmetric with atoms at $\pm\sigma$, the convolution $Q_\sigma\star N(0,1)$ has density
\begin{eqnarray*}
	\vhalf\phi(x - \sigma) + \vhalf \phi(x+\sigma) &=& \phi(x)\, e^{-\sigma^2/2}\, \cosh(\sigma x) , \\
	&=& \phi(x) \bigl(1 - \sigma^2/2 + \sigma^2 x^2/2 + o(\sigma^2) \bigr).
\end{eqnarray*}
Every symmetric distribution $P$ is a $Q_\sigma$ scale mixture.  If $P$ has finite variance~$\sigma^2$, then $P\star N(0,1)$ has exactly the same first-order small-$\sigma$ expansion as $Q_\sigma$.
In other words, every symmetric finite-variance convolutional perturbation $P\star N(0, 1)$ is first-order equivalent to the top mixture  $(1 - \theta + \theta x^2)\phi(x)$ with $\theta = \sigma^2\!/2$.  
In that sense, the class of finite-variance signals is reduced to a single mixture representative~$\Psi_2$.

The following theorem addresses the left-boundary behaviour of likelihood-based statistics for a composite mixture that includes a subset of the distributions $\Psi_\nu$.

\begin{theorem}\label{thmExample}
	Let $0 <\tau \le 2$ be given,
	and let $\hat\theta_n, \hat\nu_n$ be the maximum-likelihood estimate in the composite mixture with $\F = \{\Psi_\nu \colon 0 < \nu \le\tau\}$.
	Then, the large-sample boundary behaviour is as follows:
	\begin{eqnarray*}
		\Prob_0(\hat\theta_n > 0) &\sim& \left\{\begin{array}{ll} \tau / (2\log n) &\quad \tau < 2, \\
			1/2 & \quad \tau = 2; \end{array} \right. \\ 
		\lim_{n \to\infty} \Prob_0(\hat\nu_n = \tau \given \hat\theta_n > 0) &=& 1;\\
		\lim_{n \to\infty} \Prob_0(\Lambda_n \le x^2 \given \hat\theta_n > 0) &=& 
		\left\{\begin{array}{ll} G(x^2) &\quad \tau < 2, \\
			1 - 2\bar\Phi(x) & \quad \tau = 2. \end{array} \right.
	\end{eqnarray*}
\end{theorem}

\begin{remark}
	At the critical boundary $\theta=0$, the parameter~$\nu$ is indeterminate, which is a type of singularity that occurs in many non-regular problems.
	For $\hat\theta_n > 0$, the mixture likelihood has its maximum at the upper extremity $\hat\nu = \tau$ with high $\Prob_0$-probability because these are the distributions that are closest to Gaussian in the natural metric.  
	The limit behaviour for the composite mixture follows from the top elementary mixture $(1-\theta)\phi + \theta \psi_\tau$.
\end{remark}

\begin{proof}
	Let $\hat\theta_\nu$ denote the maximum likelihood estimator of $\theta$ at a particular value of $\nu$. Thus, $\hat\theta_n=\hat \theta_{\hat\nu}$, where $\hat\nu=\hat\nu_n$ is the maximum likelihood solution. The proof will establish containment on both sides of ${\mathcal{E}_{\hat\nu}=\{\hat\theta_{\hat\nu}>0\}}$ and ${\mathcal{E}_\tau=\{\hat\theta_\tau>0\}}$ modulo subsets of asymptotic probability 0 under $\Prob_0$, i.e.
	\begin{align*}
	\lim_{n\to\infty} \Prob_0(\mathcal{E}_{\hat{\nu}}\backslash \mathcal{E}_\tau)=\lim_{n\to\infty} \Prob_0(\mathcal{E}_{\tau}\backslash \mathcal{E}_{\hat{\nu}})=0 .
	\end{align*}
	
	Since $\psi_\nu(x) \sim K_\nu/x^{\nu+1}$ for large~$x$ and $\nu < 2$, we have
	\[
	\zeta_\nu(x) = \frac{\psi_\nu(x)} {\phi(x)} \sim \frac{\sqrt{2\pi} e^{x^2\!/2} K_\nu} { x^{\nu+1}},
	\]
	and from the argument of \S \ref{secPrelim}, the density ratio satisfies
	\[
	\Prob_0(\zeta_\nu(X) > \eta) \sim \frac{2 \psi_\nu( \sqrt{2\log \eta})}{\eta \sqrt{2\log \eta}} \sim \frac{2 K_\nu} {\eta \, (2\log\eta)^{\nu/2+1}}.
	\]
	Although the constants do not matter in this argument, we write the above upper probability, for consistency with \S \ref{secLimitDist}, in the form $\bar F(\eta)\sim 2 C_1/\eta L(\eta)$. Thus, from Corollary \ref{corollAn},
	\begin{eqnarray*}
		L(\eta) &=& (2\log\eta)^{\nu/2 + 1} / (\pi K_\nu) \sim 1/L^\dagger(\eta)\\
		B_n \sim n L^\dagger(n) &\sim& \frac{n \pi K_\nu} {(2\log n)^{\nu/2 +1}}, \\
		A_n &\sim& -4\log n/(\pi\nu), 
		\\
		T_n = B_n |A_n| &\sim& \frac{2n K_\nu} {\nu (2\log n)^{\nu/2}}.
	\end{eqnarray*}

	For sufficiently large $n$, Appendix \ref{appProfileLik} shows that $\hat\theta_n>0$ is equivalent to
	\begin{equation}\label{eqProfileLikDeriv}
	\lim_{\theta\rightarrow 0}\frac{d}{d\theta}\ell(\theta; \hat\nu(\theta))  = \lim_{\theta\rightarrow 0} \sum_{i=1}^{n}(\zeta_{\hat\nu(\theta)}(X_i)-1)  > 0,
	\end{equation}
	where $\ell(\theta; \hat\nu(\theta))$ is the profile log-likelihood function. 
	
	In view of this expression, consider the event $n^{-1}\sum\zeta_\nu(X_i)  > 1$, which corresponds to $\hat\theta_\nu > 0$ for a given $\nu$. For large~$n$, this is equivalent to $\zeta_\nu(\Xmax) > \lambda T_n$ for some $\lambda > 1$, by Theorem \ref{cjl_theorem} and ultimate monotonicity of $\zeta_{\nu}$. The condition $\zeta_\nu(X_{(n)})>\lambda T_n$ is equivalent to
	\begin{align}
	\nonumber
	X^2_{(n)}  &> 2\log (T_n X_{(n)}^{1+\nu}) + \text{const.} \\
	\nonumber
	&= 2 \log n - \nu \log (2\log n) -2 \log \nu + 2(1+\nu) \log X_{(n)} + \text{const.} \\
	\label{eq:threshold-nu}
	&= 2\log n + \nu (2\log X_{(n)} - \log (2\log n)) - 2\log \nu + 2\log X_{(n)} + \text{const.}
	\end{align}
	It follows from calculus that the right hand side above is non-increasing in $\nu$ over the interval
	\begin{align}
	\label{eq:nu-decreasing-upper}
	0< \nu \leq \left(\log \frac{X_{(n)}}{\sqrt{2\log n}}\right)^{-1}
	\end{align}
	and is otherwise increasing. Since $\tau\leq 2$, it suffices to show that, conditionally on $\hat{\theta}_n>0$, $X_{(n)}/\sqrt{2\log n}$ is bounded above by a sequence tending to 1, implying that \eqref{eq:threshold-nu} is decreasing in $\nu$ over $(0,\tau]$. To this end, define
	\begin{align*}
	\omega := \min\{x>0:x^2 \geq 2\log n + \nu(2\log x-\log(2\log n)) - 2\log\nu + 2\log x \}.
	\end{align*}
	First, we claim that for sufficiently large $n$, $\omega \leq \sqrt{2\log n + 1.01 \log (2\log n)}$. Indeed, plugging $x= \sqrt{2\log n + 1.01 \log (2\log n)}$ into the condition defining $\omega$, we see it is satisfied:
	\begin{align*}
	&2\log n + \nu [\log(2\log n + 1.01\log(2\log n)) - \log(2\log n)] - 2\log \nu + \log (2\log n) + o(1) \\
	&\sim 2\log n + (1+o(1)) \log(2\log n) \leq x^2 = 2\log n + 1.01 \log (2\log n),
	\end{align*}
	for $n$ larger than a universal constant. Next, we will show that for some $c>0$,
	\begin{align*}
	\Prob_0(X_{(n)}>\sqrt{2\log n + c\log\log n} \mid X_{(n)} > \omega) \to 0.
	\end{align*}
	Since $\omega \leq \sqrt{2\log n + 1.01\log(2\log n)}$, the left hand side above is bounded by
	\begin{align*}
	&\leq \frac{\Prob_0(X_{(n)}>\sqrt{2\log n + c\log\log n})}{\Prob_0(X_{(n)}>\sqrt{2\log n + 1.01\log(2\log n)})}\to 0,
	\end{align*}
	which follows from Mills's ratio, for any $c>1.01$, since $X_{(n)}$ is a maximum of $n$ iid standard normal random variables. This shows that conditional on $\hat{\theta}_n>0$, we have $X_{(n)} \leq \sqrt{2\log n + c \log \log n}$ with probability tending to 1, so that the right hand side of \eqref{eq:nu-decreasing-upper} diverges to $+\infty$. 
	Since the threshold is decreasing as a function of $\nu$, the set of $\nu$-values for which $\hat\theta_\nu$ is positive, i.e.,~the set of $\nu$-values for which $\zeta_\nu(\Xmax) > T_n$, must be an upper interval $(\nu_0, \tau]$, which is empty if the threshold for $\lambda=1$ is not exceeded at $\nu=\tau$. Thus, for sufficiently large $n$, $\mathcal{E}_\tau$ implies $\mathcal{E}_{\hat{\nu}}$. Conversely, given that the threshold is exceeded  on the event $\mathcal{E}_{\hat\nu}$, the approximate likelihood-ratio statistic \eqref{llr_limit} is a monotone function of the ratio
	\[
	\frac 1 {1 - R} \simeq 
	\frac{h(\Xmax)} {T_n} \simeq  \frac{\sqrt{2\pi} e^{\Xmax^2/2}} {2n\Xmax} 
	\biggl( \frac{2\log n}{\Xmax^2}\biggr)^{\nu/2} \times\nu.
	\]
	For large~$n$ and $\Xmax^2 \sim 2\log n$, the ratio increases linearly as a function of $\nu$, which implies
	$\hat\nu = \tau$ whenever $\hat\theta_n > 0$. For sufficiently large $n$, $\mathcal{E}_{\hat\nu}$ implies $\mathcal{E}_{\tau}$, and $\mathcal{E}_{\tau}$ implies $\mathcal{E}_{\hat\nu}$, thus the events are asymptotically equal under $\Prob_0$.	For $\tau < 2$, and by continuity for $\tau\le 2$, the boundary behaviour is governed by the top elementary mixture with $\nu=\tau$.
\end{proof}

\begin{remark}
	The same argument applies if we replace $\F$ with the set of Student~$t$ distributions on $\nu \le \tau$ degrees of freedom with scale parameter~$\sigma$.  In this one-dimensional composite mixture, $\sigma$~is fixed, $\tau < \infty$ is any positive real number, and the limit behavior exhibits no discontinuity at $\tau=2$ or elsewhere.
	Section~\ref{equiv_class} implies that the scale parameter is nugatory for boundary behaviour.
	This argument fails for $\tau = \infty$.
\end{remark}

\section{Predictive activity rate in multiple testing}

A basic question in multiple testing is whether all null hypotheses are true (the global null), or if at least one is false. 
Upon rejecting the global null, focus then shifts towards identifying which individual hypotheses are false. In the latter stage, it is desirable to maintain type-I error control conditional on the rejection event in the first stage. This question is relevant in settings where the signal-to-noise ratio allows for detection of a false global null, but where it is hard to identify which sites correspond to signal \citep{donoho2004higher,donoho2015special}. 
In such settings, a natural candidate for the strongest signal is the most extreme observation, and in this section we derive an asymptotically valid $p$-value for its individual null hypothesis, 
conditional on the type-I error event $\{\hat{\theta}_n>0\}$ under the global null. 

Suppose we reject the global null that $X_i \sim \mathbb{P}_0$ independently for $i=1,\dots,n$, in favour of a two-component mixture with non-null component $f_1$ and latent group indicator $\mathcal{A}_i$, %as if we knew the data were drawn from a latent variable model:
\begin{equation}
\label{eq:two-group-model}
\begin{aligned}
\mathcal{A}_i &\sim \text{Bernoulli}(\theta) \\
X_i \mid \mathcal{A}_i &\sim f_{\mathcal{A}_i} \hspace{2em} \text{$i=1,\dots,n$}
\end{aligned}
\end{equation}
concluding that $\theta>0$. Then to identify the active sites, we may compute for each observation $X_i$ the local activity rate, defined 
\begin{align*}
\mathbb{P}_\theta(\mathcal{A}_i=1 \mid X_i) := \theta f_1(X_i)/f(X_i),
\end{align*}
which measures for each observation how likely it is to be a draw from $f_1$ in the hierarchical model \eqref{eq:two-group-model}. In this section, we derive the asymptotic null distribution of the local activity rate for the observation with largest absolute value.

Given an estimate of $\theta$, the fitted or predictive activity rate for unit~$i$ is
\begin{equation}\label{eqActivityRatio}
\mathbb{P}_{\hat{\theta}_n}(\mathcal{A}_i = 1 \mid X_i) := \frac {\hat\theta_n f_1(X_{i})} {(1-\hat\theta_n)f_0(X_{i}) + \hat\theta_n f_1(X_{i})} = 
\frac {\hat\theta_n h(X_{i})} {1 - \hat\theta_n + \hat\theta_n h(X_{i})}.
\end{equation}
If $\hat\theta_n = 0$, the fitted local activity rate is zero for every unit: every unit is deemed to have a null or negligible signal. Given $\hat\theta_n > 0$, the maximum predicted activity rate is
\[
\frac {\hat\theta_n h(X_{\max})} {1 - \hat\theta_n + \hat\theta_n h(X_{\max})} = R + o_p(1)
\]
where $X_{\max}$ is the maximum absolute sample value, and $R := n\bar{Z}_n/Z_{(n)}$, where $Z_i := h(X_i)$. According to Theorem \ref{cjl_theorem}, $R$ is uniformly distributed on $(0, 1)$ in the large-sample limit when~$\theta = 0$.
%the minimum fitted local false discovery rate is $1-R$, whose conditional null distribution is also uniform on $(0, 1)$ in the large-sample  limit. 
In this case, we also have $\hat\theta_n = 0$ with probability tending to 1. That is, the type 1 error of the MLE for testing the global null $H_0: \theta=0$ tends to zero, and conditionally on the type 1 error event $\{\hat\theta_n >0\}$, the smallest fitted activity rate is an asymptotically valid $p$-value for testing $H_0$. We record this implication as a corollary of Theorem \ref{cjl_theorem} and Theorem \ref{LR_limit} below.

\begin{corollary}
	Consider the decision rule that rejects $H_0:\theta=0$ if the largest local activity rate exceeds $1-\alpha$. Then the conditional probability that some unit is declared active or non-null at level~$\alpha$ is equal to $\alpha$ in the large-sample limit. In other words, 
	\begin{align*}
	\lim_{n\to\infty}\Prob_0\left(\min_{i=1,\dots,n} \mathbb{P}_{\hat{\theta}_n}(\mathcal{A}_i=0\mid X_i)
	\leq \alpha \mathrel{\Big |} \hat{\theta}_n>0\right) = \alpha,
	\end{align*}
	for any $\alpha \in [0,1]$ where 
	\[
	\mathbb{P}_{\hat{\theta}_n}(\mathcal{A}_i=0\mid X_i) := \frac{1-\hat{\theta}_n}{1-\hat{\theta}_n + \hat{\theta}_n h(X_i)}.
	\]
\end{corollary}

\begin{proof}
	Following the proof of Theorem \ref{LR_limit}, we have $R = \hat{\theta}_n h(X_{\max})/(1+\hat{\theta}_n h(X_{\max}))$. Together with Theorem \ref{cjl_theorem}, which implies that $h(X_{\max}) \to \infty$ and $\hat{\theta}_n \to 0$ in probability conditionally on $\hat{\theta}_n>0$, we have
	\begin{align*}
	\frac {\hat\theta_n h(X_{\max})} {1 - \hat\theta_n + \hat\theta_n h(X_{\max})} &= R \times \frac{1+\hat{\theta}_n h(X_{\max})}{\hat{\theta}_n h(X_{\max})} \times \frac {\hat\theta_n h(X_{\max})} {1 - \hat\theta_n + \hat\theta_n h(X_{\max})} \\
	&= R \times \left(1 - \frac{\hat{\theta}_n}{1+\hat{\theta}_n h(X_{\max})} \right)^{-1} = R\;(1+o_p(1)).
	\end{align*}
	It follows that the complement is also Uniform$(0,1)$ distributed asymptotically, conditional on $\hat{\theta}_n>0$.
\end{proof}

\subsection*{Reproducibility} Code to reproduce all figures in the paper can be found at: \newline \href{https://www.ma.imperial.ac.uk/~hbattey/codeBMX.html}{https://www.ma.imperial.ac.uk/~hbattey/BMX.html}.

\subsection*{Acknowledgement}
P.~McC is grateful to Titus Hilberdink for advice concerning regular-variation integrals and for pointing out the connection with de Bruijn conjugates.

\bigskip
\bigskip

	\section*{Appendices}

\begin{appendix}

	\section{Calculation of the asymmetric Cauchy density}\label{appTailProbability}

	By the Fourier inversion formula of the characteristic function for $\alpha=1$ and $|\beta| \le 1$, the asymmetric Cauchy density $f(x)$ in Fig.~1 is
	\begin{align*}
	f(x) = \frac{1}{\pi} \int_0^\infty e^{-t} \cos(tx+2\beta t \log t/\pi) dt.
	\end{align*}
	Substitute $u = g_x(t) := t(1 + \frac{2\beta}{\pi x}\log t)$, which is increasing over $(\exp(-\frac{\pi x}{2\beta}),\infty)$. As $x>0$ becomes large, the integral over the interval $(0,\exp(-\frac{\pi x}{2\beta}))$ is exponentially small in $x$, and may be ignored:
	\begin{align*}
	\int_0^\infty e^{-t} \cos(x g_x(t)) dt &=  \int_{\exp(-\frac{\pi x}{2\beta})}^\infty  e^{-t} \cos(x g_x(t)) dt + O\big(e^{-\frac{\pi x}{2\beta}}\big), \\
	&= \int_{0}^\infty  e^{-g_x^{-1}(u)} \cos(x u) \frac{1}{1+\frac{2\beta\log g_x^{-1}(u)}{\pi x}+\frac{2\beta}{\pi x}} du + O\big(e^{-\frac{\pi x}{2\beta}}\big), 
	\end{align*}
	where $g_x^{-1} : (0,\infty) \to (\exp(-\frac{\pi x}{2\beta}),\infty)$ is the inverse map for the substitution. From here on, we ignore the exponentially small remainder term. It is easy to check that
	\begin{align*}
	\frac{u}{1+\frac{2\beta}{\pi x} \log u } < g^{-1}_x(u) < u,
	\end{align*}
	from which it follows that the integral is
	\begin{align*}
	&= \int_{e^{-\frac{x}{\log x}}}^{\sqrt{x}} e^{-u} \cos(x u) \frac{1}{1+\frac{2\beta\log u}{\pi x}+\frac{2\beta}{\pi x}} du (1+o(1)) + O\big(e^{-\sqrt{x}}\big), 
	\end{align*}
	where $o(1)$ means a term that goes to zero as $x \to \infty$. Then, as $x \to \infty$, the above is
	\begin{align}\label{eqMainTerm}
\nonumber	&\sim \int_{e^{-\frac{x}{\log x}}}^{\sqrt{x}} e^{-u} \cos(x u) \left( 1 - \frac{2\beta \log u}{\pi x} \right) du \\
	&= \int_{e^{-\frac{x}{\log x}}}^{\sqrt{x}} e^{-u} \cos(x u) du  - \frac{2\beta }{\pi x} \int_{e^{-\frac{x}{\log x}}}^{\sqrt{x}} e^{-u} \cos(x u) \log (u) du,
	\end{align}
    where the discarded term is negligible compared to the leading term, ultimately found to be $O(1/x^2)$, because
    \begin{align*}
        \int_{e^{-\frac{x}{\log x}}}^{\sqrt{x}}e^{-u}\cos(xu)\frac{2\beta}{\pi x} du &= \frac{2\beta}{\pi x} \int_{0}^{\infty}e^{-u}\cos(xu) du + O(e^{-\sqrt{x}}) \\
        &\sim \frac{2\beta}{\pi x} \text{Re}\Big( \int_0^\infty e^{-u(1-ix)} du \Big) \\
        &= \frac{2\beta}{\pi x} \text{Re}\Big( \frac{1}{1-ix} \Big) = \frac{2\beta}{\pi x} \text{Re}\Big( \frac{1+ix}{1+x^2} \Big) = \frac{2\beta}{\pi x(1+x^2)} = o(x^{-2}),
    \end{align*}
    where in the second line we have used $\cos(xu) = \text{Re}(e^{ixu})$. Equation \eqref{eqMainTerm} becomes
	\begin{align*}
	&= \int_{e^{-\frac{x}{\log x}}}^{\sqrt{x}} e^{-u} \text{Re}(e^{ix u}) du-\frac{2\beta }{\pi x} \int_{e^{-\frac{x}{\log x}}}^{\sqrt{x}} e^{-u} \text{Re}(e^{ixu}) \log (u) du \\
	&= \text{Re}\bigg(\frac{1}{1-ix}\int_{e^{-\frac{x}{\log x}}(1-ix)}^{\sqrt{x}(1-ix)} e^{-v} dv\bigg) -\text{Re}\bigg(\frac{2\beta }{\pi x} \int_{e^{-\frac{x}{\log x}}}^{\sqrt{x}} e^{-u(1-ix)} \log (u) du\bigg).
	\end{align*}
	The first term gives
	\begin{align*}
	\text{Re}\bigg(\frac{1}{1-ix}\int_{e^{-\frac{x}{\log x}}(1-ix)}^{\sqrt{x}(1-ix)} e^{-v} dv\bigg) &\sim \text{Re}\left(\frac{1}{1-ix} \right) = \frac{1}{1+x^2}.
	\end{align*}
	The second term is
	\begin{align*}
	&= -\text{Re}\bigg(\frac{2\beta }{\pi x} \int_{e^{-\frac{x}{\log x}}(1-ix)}^{\sqrt{x}(1-ix)} \frac{e^{-v} (\log (v) - \log(1-ix))}{1-ix} dv \bigg)\\
	&= O(x^{-3}\log x) + \underbrace{ \frac{2\beta}{\pi x}\text{Re}\bigg( \frac{1}{1-ix}\int_{e^{-\frac{x}{\log x}}(1-ix)}^{\sqrt{x}(1-ix)} e^{-v}i\sin^{-1}\bigg(\frac{-x}{\sqrt{1+x^2}} \bigg) dv\bigg)}_{(*)},
	\end{align*}
	since $1-ix = (1+x^2)^{1/2}e^{i \sin^{-1}(-x/\sqrt{1+x^2})}$. Finally, note that the above term $(*)$ is equal to
	\begin{align*}
	&= \frac{2\beta}{\pi x}\text{Re}\bigg( \frac{1+ix}{1+x^2}\int_{e^{-\frac{x}{\log x}}(1-ix)}^{\sqrt{x}(1-ix)} e^{-v} i\sin^{-1}\bigg(\frac{-x}{\sqrt{1+x^2}} \bigg) dv\bigg) \sim \frac{2\beta}{\pi x} \cdot \frac{\pi}{2x} = \frac{\beta}{x^2}.
	\end{align*}
	To summarize, we have shown that for large $x > 0$,
	\begin{align*}
	f(x) \sim \frac{1}{\pi x^2} + \frac{\beta}{\pi x^2} = \frac{1+\beta}{\pi x^2},
	\end{align*}
	which implies $\bar F(x) \sim (1+\beta)/(\pi x)$ for the right tail.
	For $x<0$, 
	\begin{align*}
	f(x) = \frac{1}{\pi}\int_0^\infty \cos(tx+2\beta t \log t/\pi)dt = \frac{1}{\pi}\int_0^\infty \cos(t|x|-2\beta t \log t/\pi)dt,
	\end{align*}
	so the argument above implies that as $x \to -\infty$,
	\begin{align*}
	f(x) \sim \frac{1-\beta}{\pi x^2},
	\end{align*}
	and $F(x) \sim (1-\beta)/(\pi|x|)$.

    	\section{De~Bruijn group}\label{appDBG}
	The de~Bruijn group is the set $(\SV\!, \diamond)$ of slowly-varying functions together with the 
	non-commutative binary operation $(L_1\diamond L_2)(x) = L_1(x)\, L_2(x L_1(x))$.
	To see that this is an associative function $\SV^2 \to \SV$, observe that
	\begin{eqnarray*}
		(L_1 \diamond (L_2\diamond L_3))(x) &=& L_1(x)\times (L_2\diamond L_3)(x L_1(x)), \\
		&=& L_1(x) \times L_2(x L_1(x))\times L_3(x L_1(x) L_2(xL_1(x)));\\
		((L_1\diamond L_2)\diamond L_3)(x) &=& (L_1\diamond L_2)(x)\times L_3(x(L_1\diamond L_2)(x)),\\
		&=& L_1(x)\times L_2(x L_1(x))\times L_3(xL_1(x) L_2(x L_1(x))), \\
		&=& (L_1 \diamond (L_2\diamond L_3))(x).
	\end{eqnarray*}
	Associativity implies that the triple product is well-defined by the pairwise products.
	The identity element: $L\diamond 1 = 1\diamond L = L$ is the unit constant function.
	If it exists, the inverse is a slowly-varying function $L^\dagger$ such that
	$(L\diamond L^\dagger)(x) =  (L^\dagger\diamond L)(x) =  1$ for all~$x > 0$.
	Since slow variation is a characterization of the limiting behaviour for large~$x$,
	it says little about the behaviour for general~$x$ beyond continuity or measurability.
	Thus, the existence of an inverse is not guaranteed.
	Nevertheless, the following theorem suffices for present purposes.
	
	\begin{theorem}[de Bruijn, 1959]
		To each $L\in\SV$ there corresponds a function $L^\dagger \in \SV\!$, 
		satisfying $L(x) L^\dagger(x L(x)) = L^\dagger(x) L(xL^\dagger(x)) = 1$ for all~$x$ sufficiently large.
	\end{theorem}

	\begin{proof}
		Theorem~1.5.13 of Bingham, Goldie and Teugels (1987) proves existence in the sense of equivalence.
		Theorem~1.8.9 proves existence as stated above, i.e.,~for all~$x$ sufficiently large.
	\end{proof}
	
	Although the product $xL(x)$ is ultimately monotone, it is not necessarily monotone for small~$x$.
	Thus, $\SV$ contains functions for which no group inverse exists (as a function $(0,\infty) \to (0,\infty)$).
	The theorem states that an asymptotic inverse exists, which implies that the set of equivalence classes
	$(\SV, \diamond)/{\sim}$ is a group.
	For the most part, it is the group of equivalence classes that is of interest here.
	However, we always work with a representative element.
	
	\section{Approximation of the sine-integral}\label{appSinIntegral}
	Let $F$ be the cumulative function of a finite-mean distribution on the positive real line, 
	and let $\bar F(x) = 1/(x L(x))$ be the right-tail probability.
	Since $L(\lambda x)/L(x) \to 1$ uniformly in $\lambda$ as $x\to\infty$, the derivative $x L'(x)/L(x)$ 
	(with respect to~$\lambda$ at $\lambda=1$) tends to zero.
	Hence
	\[
	dF(x) = - d\bar F(x) = \biggl(1 + \frac{xL'(x)} {L(x)} \biggr) \frac{dx} {x^2 L(x)} \sim \frac{dx} {x^2 L(x)}.
	\]
	
	We first split the integrand into two parts, $\sin(tx) = t x + (\sin t x - t x)$,
	and split the range into disjoint intervals $(0, T)$ and $(T, \infty)$, where $T = 1/|t|$ is large.
	Then
	\begin{equation}\label{split_integral}
	\int_0^\infty\!\! \sin(tx)\, dF(x) = t \mu + \int_0^T\!\! (\sin t x - tx)\, dF(x) + \int_T^\infty\!\! (\sin(tx) - tx) \, dF(x),
	\end{equation}
	where $\mu$ is the mean.
	The first integrand is bounded by $|\sin(tx) - tx| \le |t|^3 x^3$ for $x \le T$, so the contribution from the sub-interval
    $x \le T^{1/2}$ is at most $T^{-3/2}$, while the remainder is bounded by
	\begin{eqnarray*}
		|t|^3 \int_{T^{1/2}}^T x^3\, dF(x) &=& |t|^3 \int_{T^{1/2}}^T x^3 \Big(1 + \frac{x L'(x)} {L(x)} \Bigr)\frac{dx} {x^2 L(x)} , \\
		&\sim& |t|^3 \int_{T^{1/2}}^T x^3 \frac{dx} {x^2 L(x)} , \\
		& \le & |t|^3 T \frac T {L(T)} = \frac 1 {T L(T)}.
	\end{eqnarray*}

	Since $|\sin(tx)| \le 1$, the first part of the second integral in \eqref{split_integral} has the same bound:
	\[
	\Big | \int_T^\infty \sin(tx)\, dF(x) \Big| \le \int_T^\infty dF(x) = \bar F(T) = \frac 1 {T L(T)}.
	\]
	It follows that
	\begin{eqnarray*}
		\int_0^\infty \sin(tx)\, dF(x) &=& t \mu  - t \int_T^\infty  x\, dF(x) + O\biggl( \frac t { L(T)} \biggr), \\
		&=& t\mu - t\int_T^\infty \frac{dx} {x L(x)} + O\biggl( \frac t { L(T)} \biggr) .
	\end{eqnarray*}
	
	\section{Refined asymptotic solution for Example \ref{exampleGaussCauchy2}}\label{appRefined}

The density ratio from Example \ref{exampleGaussCauchy2} is
\[
h(x) = \frac{\sqrt{2\pi} e^{x^2/2}}{\pi (1+x^2)}
\]
From the first and second derivatives of $h$, the stationary points are at $x=0$, $x=\pm 1$ and the minimum value is $h(1)=(e/2\pi)^{1/2}$. We seek an approximate positive solution $\xi(\eta)$ in $\xi$ of the equation $h(\xi)=\eta$. Equivalently, on multiplying both sides by $\sqrt{2\pi e}$ and taking logarithms, $\xi(\eta)$ solves $w=T(w)$, where $w=\xi^2+1$ and
\[
T(w) =2 \log(\eta \sqrt{2\pi e}/2) + 2\log(w).
\]
Write $w(\eta)$ for this solution in $w$, i.e.~$w(\eta)=\xi(\eta)^2+1$. For arbitrary $w,w^\prime>1$, $|T(w)-T(w')|=(2/\overline{w})|w-w'|$, where $\overline{w}\in(w,w')$. Thus $T(w)$ is a contraction mapping over $(2,\infty)$. With $w_0 = 2\log(\eta\sqrt{2\pi e}/2)$,
\begin{eqnarray*}
	w_1 &=& 2\log (\eta\sqrt{2\pi e}/2) + 2\log(w_0) \\
	&=& 2\log(\eta\sqrt{2\pi e}) + 2\log (\log(\eta\sqrt{2\pi e}/2)),
\end{eqnarray*}
is an approximation to $w(\eta)$ with error $|w(\eta)- w_1|\leq q|w_1-w_0|/(1-q)$, where $q=O(1/w(\eta))$ ($w\rightarrow \infty$). It follows that 
\[
w(\eta)= 2\log(\eta\sqrt{2\pi e }) + 2\log(\log(\eta\sqrt{2\pi e }/2)) + O\Bigl(\frac{\log\log \eta}{\log \eta}\Bigr), \quad (\eta\rightarrow \infty).
\]
Thus, $\xi(\eta)$ has asymptotic solution
\[
\vhalf \xi^2 = \log \eta + \log(\log \eta) + o(\log(\log \eta)).
\]
This gives rise to an expression for $L(\eta)=\xi^{-1}f_1(\xi)$ of the form $L(\eta)=(2 \log (\eta\log\eta))^{3/2}$, necessitating a refinement of Theorem \ref{thmSinIntegral}. Specifically, for 
$L(x)=(\beta_0 (\log(x\log x)))^{\delta+1}$ the relevant integral in the proof of Theorem \ref{thmSinIntegral} becomes
\begin{eqnarray*}
	\int_T^\infty \frac{dx}{xL(x)} &=&
	\int_T^\infty \frac {dx}{x\, (\beta_0\log (x\log x))^{\delta + 1}} = \frac{1}{\beta_0^{\delta+1}}\int_{\log T}^\infty \frac {du}{(u + \log u)^{\delta + 1}} \\
    &\sim& \frac{1}{\beta_0^{\delta+1}} \int_{\log(T\log T)}^{\infty}\Bigl(\frac{q - \log q}{q-\log q + 1} \Bigr)\frac{dq}{q^{(\delta+1)}} \sim 
	\frac{(\log (T\log T))^{2\delta}}{\delta L(T)},
\end{eqnarray*}
where in the second transformation we have used Lagrange inversion \citep[][p.~25]{deBruijn1958} for $v:=e^q=ue^u$ in the form $u\sim \log v - \log \log v = q-\log q$. The appropriate modifications to Corollaries \ref{corollAn} and \ref{corollXbar} then give 
\[
\PP_0(\hat \theta_n >0) \sim (2(\log(n\log n)))^{-1}
\]
which closely approximates the empirical probabilities (see Fig.~\ref{FigSim}a).

	\section{Proof of Proposition \ref{prop:tail-inheritance}}\label{app:tail-prop}

    \begin{proof}
The convolution density at $y > 0$ is
\[
f_1(y) = \int_\Real  \phi(y - x) \, G(dx) = \int_{|x|\le y^{1/2}} \phi(y - x) \, G(dx) 
	+ \int_{|x| > y^{1/2}}\phi(y - x) \, G(dx).
\]
For large positive $y$, the first integral is bounded by $\phi(y - y^{1/2})$, and thus negligible.
In the second integral, the contribution from $x < -y^{1/2}$ is negligible compared with
the contribution from the positive interval $x > y^{1/2}$.
Using the regular-variation approximation for large~$y$, the log integrand 
$h(x) = -\vhalf (y - x)^2 - x^\gamma L(x)$ has a unique maximum at
\[
\hat x = y - y^{\gamma-1} \bigl(\gamma L(y) - y L'(y) \bigr) \sim y.
\]
Since the negative second derivative of the integrand is asymptotic to one, 
the dominant contribution for large~$y$ is the Laplace approximation
\[
-\log f_1(y) \sim-\log \int _{x > y^{1/2}} \phi(x - y)\, g(x) \, dx \sim -\log g(\hat x) \sim -\log g(y),
\]
where $g$ is the density of $G$. The Laplace argument can be modified for $\gamma = 2$, but the conclusion is different.
\end{proof}

	\section{Derivation of equation \eqref{eqProfileLikDeriv}}\label{appProfileLik}
	
	For a log-likelihood function of the form 
	\[
	\ell(\theta, \nu) = \sum_{i=1}^n \log\bigl((1-\theta)f_0(X_i) + \theta f_1(X_i; \nu)\bigr),
	\]
	the profile log-likelihood derivative is 
	\begin{equation}\label{eqDerivPLL}
	\frac{d}{d\theta}\ell(\theta; \hat\nu(\theta)) = \sum_{i=1}^n \frac{(h(X_i;\hat\nu(\theta))-1) + \theta t(X_i;\hat\nu(\theta))}{1+\theta(h(X_i;\hat\nu(\theta))-1)},
	\end{equation}
	where $h(x;\nu)=f_1(x/\nu)/f_0(x)$ and 
	\begin{equation}\label{eqt}
	t(x;\hat\nu(\theta))=\frac{1}{f_0(x)}\frac{\partial f_1(x;\hat\nu(\theta))}{\partial \hat\nu} \frac{d\hat\nu(\theta)}{d\theta}.
	\end{equation}
	and for any given $\theta$, $\hat\nu(\theta)$ solves 
	\[
	F(\theta,\hat\nu(\theta)):= \frac{\partial \ell(\theta,\nu)}{\partial \nu}\biggr|_{\nu=\hat\nu(\theta)} =0.
	\]
	Total differentiation of this equation $F(\theta,\hat\nu(\theta))=0$ gives
	\begin{equation}\label{dtheta}
	\frac{d\hat\nu(\theta)}{d\theta} = -\frac{j_{\nu\theta}(\theta,\hat\nu(\theta))}{j_{\nu\nu}(\theta,\hat\nu(\theta))},
	\end{equation}
	where $j_{\nu\theta}(\theta,\hat\nu(\theta))$ and $j_{\nu\nu}(\theta,\hat\nu(\theta))$ are the cross and double derivatives of the joint log-likelihood function evaluated at $(\theta,\hat\nu(\theta))$, given in the present context of $f_1(x;\nu)=\psi_\nu(x)\sim K_\nu/|x|^{\nu+1}$ by
	\[
	j_{\nu\theta}(\theta,\hat\nu(\theta)) = \sum_{i=1}^n \frac{d \psi_\nu (X_i)}{d\nu} \biggl|_{\nu=\hat\nu(\theta)}, \quad j_{\nu\nu}(\theta,\hat\nu(\theta)) = \theta \sum_{i=1}^n \frac{d^2 \psi_\nu (X_i)}{d\nu^2} \biggl|_{\nu=\hat\nu(\theta)},
	\]
	where
	\begin{eqnarray*}
		\psi'_\nu(x) :=	\frac{d \psi_\nu(x)}{d\nu}  &\sim & \frac{K_\nu' - K_\nu \log |x|}{|x|^{\nu+1}}\\
		\psi''_\nu(x) :=	\frac{d^2\psi_\nu(x)}{d\nu^2} &\sim & \frac{K_\nu'' - 2K_\nu'\log |x| + K_\nu(\log |x|)^2}{|x|^{\nu+1}}.
	\end{eqnarray*}
	On substituting \eqref{dtheta} in \eqref{eqDerivPLL} via \eqref{eqt}, and on making the replacement $f_0(x)=\phi(x)$, the profile log-likelihood derivative is
	\begin{equation}\label{eqProfileLLDeriv}
	\frac{d}{d\theta}\ell(\theta; \hat\nu(\theta)) 
	=\sum_{i=1}^{n}\frac{(\zeta_{\hat\nu(\theta)}(X_i) - r_{\hat{\nu}(\theta)}(X_i)-1)}{{1+\theta(\zeta_{\hat\nu(\theta)}(X_i)-1)}},
	\end{equation}
	where
	\begin{eqnarray*}
		r_{\hat{\nu}(\theta)}(X_i)&=&\frac{\psi'_{\hat\nu(\theta)}(X_i)}{\phi(X_i)}\frac{\sum_{j=1}^n \psi'_{\hat\nu(\theta)}(X_j)}{\sum_{j=1}^n \psi''_{\hat\nu(\theta)}(X_j)} \\
		&\sim& \frac{K_{\hat{\nu}(\theta)} \log |X_i|}{|X_i|^{\hat{\nu}(\theta)+1} \phi(X_i)}\biggl(\frac{\sum_{j=1}^n K_{\hat{\nu}(\theta)} \log |X_j|/|X_j|^{\hat{\nu}(\theta)+1}}{\sum_{j=1}^n K_{\hat{\nu}(\theta)}(\log |X_j|)^2/|X_j|^{\hat{\nu}(\theta)+1}}\biggr)\rightarrow_{\Prob_0} 0, \quad (n\rightarrow \infty).
	\end{eqnarray*}
	Taking the limit of \eqref{eqProfileLLDeriv} as $\theta\rightarrow 0$ gives \eqref{eqProfileLikDeriv}. \qed
	
\end{appendix}

\bigskip
\bigskip

\end{document}